%% file: DNA_NIPS_Arxiv.tex
\documentclass{article}

\input{everything.tex}
\newcommand{\compactify}{} 

\usepackage[noend]{algpseudocode}
\usepackage{verbatim}
\usepackage[ruled,vlined]{algorithm2e}
\usepackage{wrapfig}

\usepackage{nicefrac}

 \usepackage{xcolor}
\usepackage{color}
\usepackage{graphicx}
\usepackage{fullpage}

\usepackage{subcaption}
\newcommand{ \kRNA}{DNA\xspace}

\usepackage[round]{natbib}

\usepackage[final]{neurips_2019}

\usepackage{hyperref}       
\usepackage{url}            
\usepackage{booktabs}       
\usepackage{amsfonts}       
\usepackage{nicefrac}       
\usepackage{microtype}      

\title{Direct Nonlinear Acceleration}

\author{
Aritra Dutta${}^1$ \;
El Houcine Bergou${}^{1,2}$ \; Yunming Xiao${}^1$ \; Marco Canini${}^1$ \; Peter Richt\'{a}rik${}^{1,3}$ \\
${}^1$ KAUST, Saudi Arabia \; ${}^2$ INRA, Universit\'e Paris-Saclay, France \; ${}^3$ MIPT,  Russia}

\begin{document}

\maketitle

\begin{abstract}
Optimization acceleration techniques such as momentum play a key role in state-of-the-art machine learning algorithms. Recently, generic vector sequence extrapolation techniques, such as  regularized nonlinear acceleration (RNA) of Scieur et al.\ \citep{rna_16}, were proposed and shown to accelerate  fixed point iterations. In contrast to RNA which computes extrapolation coefficients by (approximately) setting the gradient of the objective function to zero  at the extrapolated point, we propose a more direct approach, which we call {\em direct nonlinear acceleration (DNA)}. In DNA, we aim to minimize (an approximation of) the function value at the extrapolated point instead. We adopt a regularized approach with regularizers  designed  to prevent the model from entering a region in which the functional approximation is less precise. ~While the computational cost of DNA is comparable to that of RNA, our direct approach significantly outperforms RNA on both synthetic and real-world datasets. While the focus of this paper is on convex problems, we obtain very encouraging results in accelerating the training of neural networks.
\end{abstract}
\section{Introduction}
 In this paper we consider the generic unconstrained minimization problem
 \begin{equation}\label{eq:mainP}
\compactify \min_{x\in \R^n} f(x),
 \end{equation}
where $f: \R^n \rightarrow \R$ is a smooth  objective function and bounded from below.   One of the most fundamental methods for solving  \eqref{eq:mainP} is {\em gradient descent (GD)}, on which many state-of-the-art methods are based. Given current iterate $x_k\in \R^n$,  the update rule of GD is
  \begin{eqnarray}\label{GD}
\compactify x_{k+1} = x_k - \alpha_k \nabla f(x_k),
 \end{eqnarray}
 where $\alpha_k>0$ is a stepsize. The efficiency of GD depends on further properties of $f$. Assuming $f$ is $L$--smooth and $\mu$--strongly convex, for instance, the iteration complexity of GD is $\cO(\kappa{\rm log}(1/\epsilon))$, where $\kappa=L/\mu$ and $\epsilon$ is the target error tolerance. However, it is known that GD is not the ``optimal'' gradient type method: it can be {\em accelerated}.  
 
 
The idea of accelerating converging optimization algorithms can track its history back to 1964 when Polyak proposed his ``heavy ball'' method \citep{polyak}. 
In 1983, Nesterov proposed his accelerated version for general convex optimization problems. Comparing with Polyak's method, Nesterov's method gives acceleration for 
general convex and smooth problems and the iteration complexity improves 
to  ${\cO(1/\sqrt{\epsilon})}$ \citep{nesterov}. In 2009, Beck and Teboulle proposed fast iterative shrinkage thresholding algorithm (FISTA) \citep{fista} that uses Nesterov's momentum coefficient and accelerates {\em proximal} type algorithms to solve a more complex class of objective functions that combine a smooth, convex loss function (not necessarily differentiable) and a strongly convex, smooth penalty function (also see \citep{nesterov2007, Nesterov2013}). To develop further insights into Nesterov's method, Su et al.\ \citep{ode_boyd} examined a continuous time 2nd-order ODE which at its limit reduces to Nesterov's accelerated gradient method. In addition, Lin et al. \citep{lin_katalyst} introduced a generic approach known as {\em catalyst} that minimizes a convex objective function via an accelerated proximal point algorithm and gains acceleration in Nesterov's sense. \citep{BubeckLS15} proposed a geometric alternative to gradient descent that is inspired by ellipsoid method and produces acceleration with complexity $\compactify{\cO(1/\sqrt{\epsilon})}$. Recently, \citep{Zhu2017LinearCA} used a linear {\em coupling} of gradient descent and mirror descent and claimed to attend acceleration in Nesterov's sense as well. 
In contrast, the {\em sequence acceleration} techniques accelerate a sequence independently from the iterative method that produces this sequence. 
In other words, these techniques take a sequence $\{x_k\}$ 
and produce an accelerated sequence based on the linear combination of $x_k{\rm s}$ such that the new accelerated sequence converges faster than the original. 
In the same spirit, recently, Scieur et al.~\citep{rna_16,rna_18} proposed an acceleration technique called regularized nonlinear acceleration~(RNA). 
Scieur et al.'s idea is based on Aitken's $\Delta^2$-algorithm \citep{aitken} and Wynn's $\epsilon$-algorithm \citep{wynn} (or recursive formulation of generalized Shanks transform \citep{ shanks, wynn, brezinski}). 
To achieve acceleration, 
Scieur et al. considered a technique known as minimum polynomial approximation and they assumed a {\em linear} model for the iterates near the optimum. They also proposed a {\em regularized} variant of their method to stabilize it numerically. 
The intuition behind the regularized nonlinear acceleration of Scieur et al. is very natural. To minimize $f$ as in \eqref{eq:mainP}, they considered the sequence of iterates $\{x_k\}_{k\geq 0}$ is generated by a fixed-point map. If $x^\star$ is a minimizer of $f$, $\nabla f(x^\star)=0$, and hence through extrapolation one can find:
\begin{eqnarray}\label{eq:rna0}
\compactify  c^\star \approx \arg\min_{c} 
\left\{ 
\left\|
\nabla f\left(\sum_{k = 0}^K c_k x_k \right) 
\right\| 
\;:\; c\in\R^{K+1}, \;\sum_{k = 0}^K c_k =1   
\right\},
 \end{eqnarray} 
such that the next (accelerated) point can be generated as a linear combination of $K+1$ previous iterates:
$
x = \sum_{i=0}^K  c_i^\star x_i.
$
 We review RNA in detail in Section \ref{sec:rna}. 

{\em Notation.} We denote the $\ell_2$-norm of a vector $x$ by $\| x \|$  and define $\| x \|_M$ by $\|x\|_M\eqdef \sqrt{x^\top M x}.$

\subsection{Contributions}
We highlight our main contributions in this paper as follows: 

\textbf{Direct nonlinear acceleration~(DNA).} Inspired by Anderson's acceleration technique \citep{Anderson} (see Appendix for a brief description of Anderson's acceleration) and the work of Scieur et al. \citep{rna_16}, we propose an extrapolation technique that accelerates a converging iterative algorithm. However, in contrast to \citep{rna_16}, we find the extrapolation coefficients $c^\star$ by directly minimizing the function at the linear combination of $K+1$ iterates $\{x_k\}_{k\geq 0}^K$ with respect to $c\in\R^{K+1}$. In particular, for a given sequence of iterates $\{x_k\}_{k\ge 0}^K$ we propose to approximately solve:
\begin{eqnarray}\label{mainDNA}
\boxed{\compactify \min_{c\in\R^{K+1}} f\left(\sum_{k = 0}^K c_k x_k \right)+\lambda  g(c),}
\end{eqnarray}
where  
$\lambda>0$ is a balancing parameter and 
$g$ is a penalty function. As our approach tries to minimize the functional value directly, we call it as \emph{direct nonlinear acceleration} (DNA). We also note that our formulation shares some similarities with \citep{riseth,zhang2018globally}. However, unlike \citep{riseth}, we do not require line search and check a decrease condition at each step of our algorithm. On the other hand, Zhang et al.\ \citep{zhang2018globally} do not consider a direct acceleration scheme as they deal with a fixed-point problem.

\textbf{Regularization.}
We propose several versions of DNA by varying the penalty function $g(c)$. This helps us to deal with the numerical instability in solving a linear system as well as to control  errors in gradient approximation. In our first version, we let $g(c) =  1_S(c)$, where $S\eqdef \{c\;:\; \sum_i c_i=1\}$ and $1_{S}(c)=0$ if $c\in S$, while $1_{S}(c)=+\infty$ otherwise.  Later, we propose two regularized {\it constraint-free} versions to find a better minimum of the function $f$ by expanding the search space of extrapolating coefficients to $\R^{K+1}$ rather than restricting them over the space $S$. To this end, the first {constraint-free version} 
adds a quadratic regularization $\compactify{ g(c) =  \left\|\sum_{i=0}^K c_ix_i-y\right\|^2}$ to the objective function, where  $y$ is a reference point and $g(c)$ 
controls how far we want the linear combination $\sum_ic_ix_i$ to deviate from $y$. In the second {constraint-free version}, we add the regularization directly on $c$. We add a quadratic term of the form  $g(c) =\compactify{ \|c-e\|^2}$ to the objective function, where $e$ is a reference point to $c$ and $g(c)$ controls how far we want $c$ to deviate from $e$. In contrast, the {\em regularized} version of RNA only considers a ridge regularization $\|c\|^2$ for numerical stability. Trivially, we note that by setting $e=0$, we recover the regularization proposed in RNA.  We argue that by using a different penalty function $g(c)$ as regularizer our DNA is more robust than RNA. 

\textbf{Quantification between RNA and DNA in minimizing quadratic functions by using GD iterates.}
%
 If $g(c)=0$ or $g(c)=1_{\sum_ic_i=1}$, in terms of the functional value, we always obtain a better accelerated point than 
RNA. Moreover, the acceleration obtained by DNA can be {\em theoretically} directly implied from the existing results of Scieur et al.\citep{rna_16}.
If $g(c)=0$, we show by a simple example on quadratic functions that DNA outperforms RNA by an arbitrary large margin.  
If $g(c)=1_{\sum_ic_i=1}$, we also quantify the functional values obtained from both RNA and DNA for quadratic functions and provide a bound on how DNA outperforms RNA in this setup.

\textbf{Numerical results.}  Our empirical results show that for smooth and strongly convex functions, minimizing the functional value converges faster than RNA. In practice, our acceleration techniques are robust and outperform that of Scieur et al.~\citep{rna_16} by large margins in almost all experiments on both synthetic and real datasets. To further push the robustness of our methods, we test them on nonconvex problems as well. As a proof of concept, we trained a simple neural network classifier on MNIST dataset \citep{mnist} via GD and accelerate the GD iterates via the online scheme in \citep{rna_16} for both RNA and DNA. Next, we train ResNet18 network \citep{resnet18} on CIFAR10 dataset \citep{cifar10} by SGD and accelerate the SGD iterates via the online scheme in \citep{rna_16} for both RNA and DNA. In both cases, DNA outperform RNA in lowering the generalization errors of the networks.


 \section{Regularized Nonlinear Acceleration}\label{sec:rna}
In RNA, one solves  \eqref{eq:rna0} by assuming that the gradient can be approximated by linearizing it in the neighborhood of $\{x_k\}_{k=0}^K.$
Thus, by assuming $\sum_{k = 0}^K c_k =1$, the relation
$
\compactify{ \left\|\nabla f\left(\sum_{k = 0}^K c_k x_k \right)\right\|\approx \left\| \sum_{k = 0}^K c_k\nabla f\left(x_k \right)\right\|}
$
 holds. Hence, one can approximately solve  \eqref{eq:rna0} via:
 \begin{eqnarray}\label{eq:rna} 
\compactify  c^\star=\arg\min_{c} \left\{ \left\| \sum_{k = 0}^K c_k \nabla f\left(x_k \right)\right\| =\left\| \sum_{k = 0}^K c_k \tilde{R}_k\right\|\;:\;  c\in\R^{K+1}, \; \sum_{k = 0}^K c_k =1 \right\},
 \end{eqnarray} 
 where $\tilde{R}_k$ is the $k^{\rm th}$ column of the matrix $\tilde{R}$, which holds $\nabla f\left(x_k \right).$~Moreover \eqref{eq:rna} does not need an explicit access to the gradient and it can be seen as an approximated minimal polynomial extrapolation (AMPE) as in \citep{ampe,rna_16,rna_18}. 
 In this context, we should note that Scieur et al. indicated that the summability condition $c^\top\mathbbm{1}=1$ is not restrictive, where $\mathbbm{1}$ is a vector of all 1s. 
If the sequence  $\{x_k\}$ is generated via GD (as in \eqref{GD}), then $\tilde{R} =\left[\nicefrac{(x_0 - x_{1})}{\alpha_0},\ldots,\nicefrac{(x_K - x_{K+1})}{\alpha_K}  \right]$. 
Also, if $\tilde{R}^\top \tilde{R}$ is nonsingular, then the minimizer of \eqref{eq:rna} 
is explicitly given as:
$\compactify{c^\star = \frac{(\tilde{R}^\top \tilde{R})^{-1}\mathbbm{1}}{\mathbbm{1}^\top(\tilde{R}^\top \tilde{R})^{-1}\mathbbm{1}}}$.  If $\tilde{R}^\top \tilde{R}$ is singular then $c$ is not necessarily unique. Any $c$ of the form $\frac{z}{z^\top\mathbbm{1}}$, where $z$ is a solution of $\tilde{R}^\top \tilde{R} z = \mathbbm{1}$, is a solution 
 of (\ref{eq:rna}). To deal with the numerical instabilities and the case when the matrix 
 $\tilde{R}^\top \tilde{R}$ is singular, Scieur et al. proposed to add a regularizer of the form
 $\lambda \|c\|^2$ to their problem, where $\lambda>0$. As a result, $c^\star$ is unique and given as 
$\compactify{c^\star = \frac{(\tilde{R}^\top \tilde{R}+\lambda I)^{-1}\mathbbm{1}}{\mathbbm{1}^\top(\tilde{R}^\top \tilde{R}+\lambda I)^{-1}\mathbbm{1}}}.$
   The numerical procedure of RNA is given in Alg~\ref{alg:RNA}.
 For further details about RNA we refer the readers to \citep{rna_16,rna_18}. Scieur et al. also explained several acceleration schemes to use with Algorithm~\ref{alg:RNA}. 
\begin{algorithm}
	\SetAlgoLined
 	\SetKwInOut{Input}{Input}
	\SetKwInOut{Output}{Output}
     \SetKwInOut{Init}{Initialize}
     \SetKwInOut{Compute}{Compute}
\Input{Sequence of iterates $x_0,\ldots,x_{K+1}$; sequence of step sizes $\alpha_0,\ldots,\alpha_{K}$; $\mathbbm{1}\in\R^{K+1}$: a vector of all 1s; and $\lambda >0$.}
		\nl Set $\tilde{R} =\left[\frac{x_0 - x_{1}}{\alpha_0},\ldots,\frac{x_K - x_{K+1}}{\alpha_K}  \right] $\;
		\nl Solve the linear system: $\left(\tilde{R}^\top \tilde{R} + \lambda I \right)z = \mathbbm{1}$\;
		\nl Set $c = \frac{z}{z^\top\mathbbm{1}}\in\mathbb{R}^{K+1}$\;
       \Output{$x = \sum_{k = 0}^K c_k x_k$.}
 	\caption{RNA} \label{alg:RNA}
 \end{algorithm}


\section{ Direct Nonlinear Acceleration}\label{sec:dna}
Instead of minimizing the norm of the gradient, we propose to minimize the objective function $f$ directly to obtain the coefficients $\{c_k\}$. 
We set $g(c)=0$ in \eqref{mainDNA} and we propose to  solve the unconstrained minimization problem 
\begin{eqnarray}\label{eq:fkrna}
\compactify \boxed{ \min_{c\in\R^{K+1}}  f\left(Xc \right),}
 \end{eqnarray}
 where $X = [x_0, \ldots,x_K]$. We call problem \eqref{eq:fkrna} as direct nonlinear acceleration (DNA) without any constraint.
 If $f$ is quadratic, then we have the following lemma:
 \begin{lemma}\label{lemma:dna1}
Let the objective function $f$ be quadratic and let $\{x_k\}$ be the iterates produced by \eqref{GD} to minimize $f$. 
 Then 
  $c$ is a solution of the linear system $X^\top R z = - X^\top \nabla f(0),$
  where  $R \in\R^{n\times (K+1)}$ is a matrix such that its $i^{\rm th}$ column is  
$R_i = \frac{x_i-x_{i+1}}{\alpha_i}-\nabla f(0)$ and $X = [x_0, \ldots,x_K]$.
 \end{lemma}
 If $f$ is non-quadratic then we can {\em approximately} solve problem (\ref{eq:fkrna})
 by approximating its gradient by a linear model.
 In fact, we use the following approximation $\nabla f(x)\approx A(x-y_x)+\nabla f(y_x),$ where we assume that $x$ is close to $y_x$ and $A$ is an approximation of the Hessian. Therefore, by setting $x=Xc$ and $y_x=y$ in the above, we have
$
\nabla f(Xc) \approx A(Xc-y)+\nabla f(y)=\sum_i c_iAx_i - A y +\nabla f(y),
$
where $y$ is a referent point that is assumed to be in the neighborhood of $Xc$. For instance, one may choose $y$ to be $x_K$. Let $x_{i-1}$ be a referent point for $x_i$, that is, assume that $\nabla f(x_i) \approx A(x_{i}-x_{i-1})+\nabla f(x_{i-1})$. Then one can show that 
 $Ax_i = \nabla f(x_i)-\nabla f(0)$. As a result, we have
\begin{eqnarray}\label{approximate_grad}
 \compactify \nabla f(Xc)& \approx &\compactify \sum_i c_i (\nabla f(x_i)-\nabla f(0))- A y +\nabla f(y)\nonumber\\
&=&\compactify \sum_i c_i \left({\tfrac{x_i-x_{i+1}}{\alpha_i}}-\nabla f(0)\right)- A y +\nabla f(y)\nonumber\\
&=&\compactify \sum_i c_iR_i- A y +\nabla f(y) \quad = \quad Rc - A y +\nabla f(y) \quad  \approx \quad  Rc  +\nabla f(0).
\end{eqnarray}
Therefore, from the first optimality condition and by using 
(\ref{approximate_grad}), we conclude that the solutions of (\ref{eq:fkrna})
can be approximated by the solutions of the linear system $X^\top R z = - X^\top \nabla f(0)$.
 We describe the numerical procedure in Alg~1 in the Appendix.
 \begin{figure}
\centering
\begin{minipage}{0.3\textwidth}		
\centering
		\includegraphics[width=\textwidth]{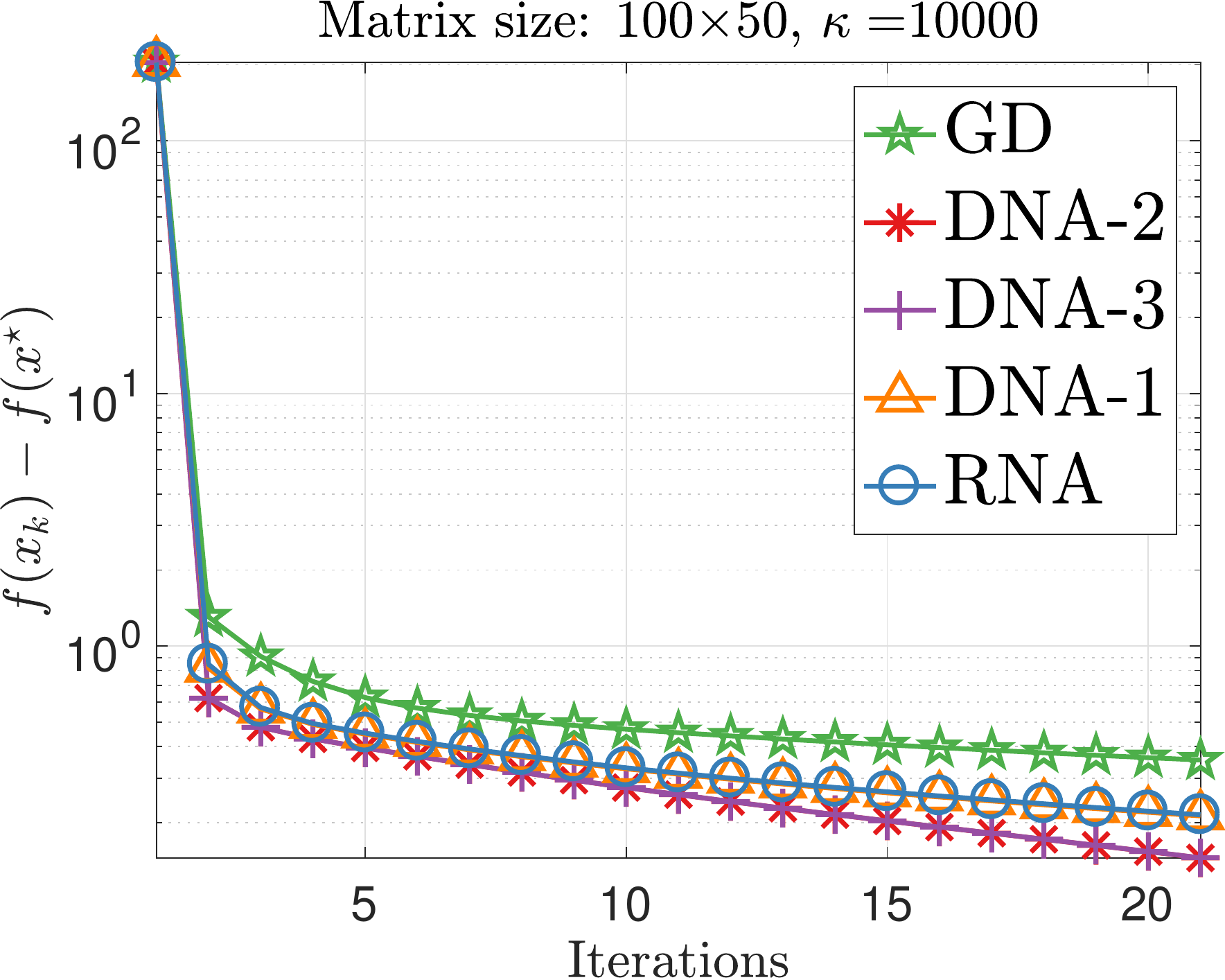}
	\end{minipage}
\begin{minipage}{0.3\textwidth}
		\centering
		\includegraphics[width=\textwidth]{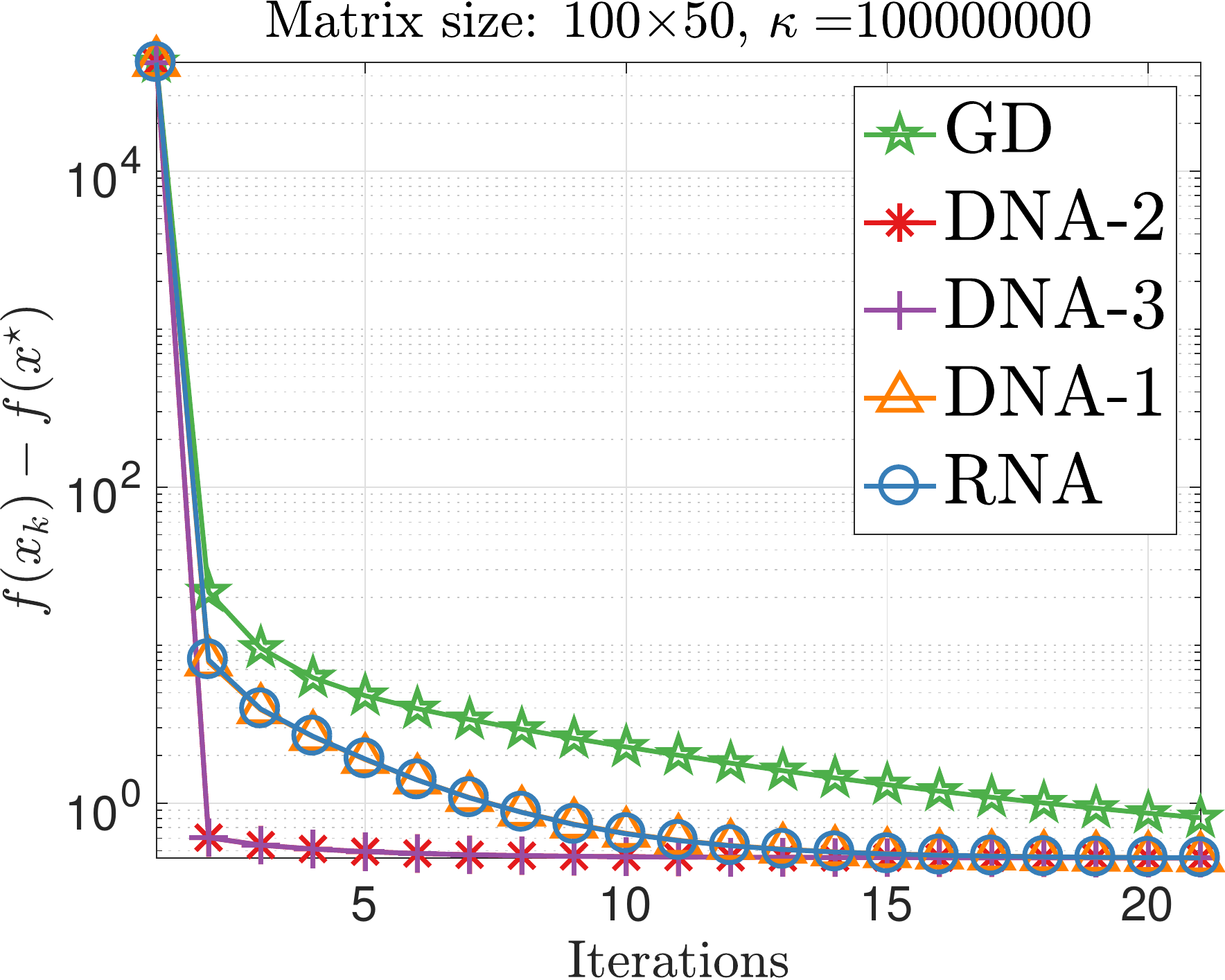}
	\end{minipage}
\begin{minipage}{0.3\textwidth}
		\centering
		\includegraphics[width=\textwidth]{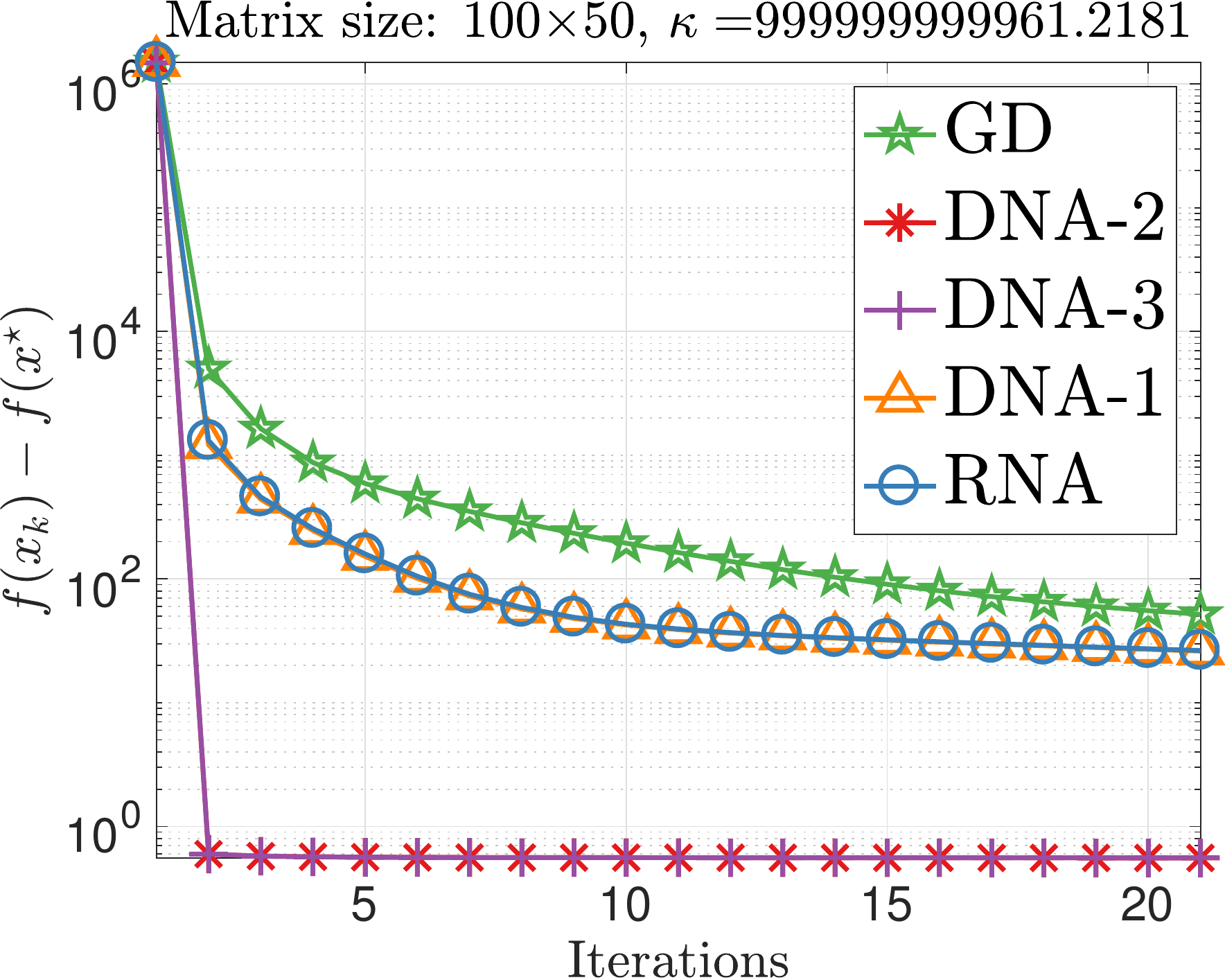}
	\end{minipage}
     \\
	\begin{minipage}{0.3\textwidth}
		\centering
		\includegraphics[width=\textwidth]{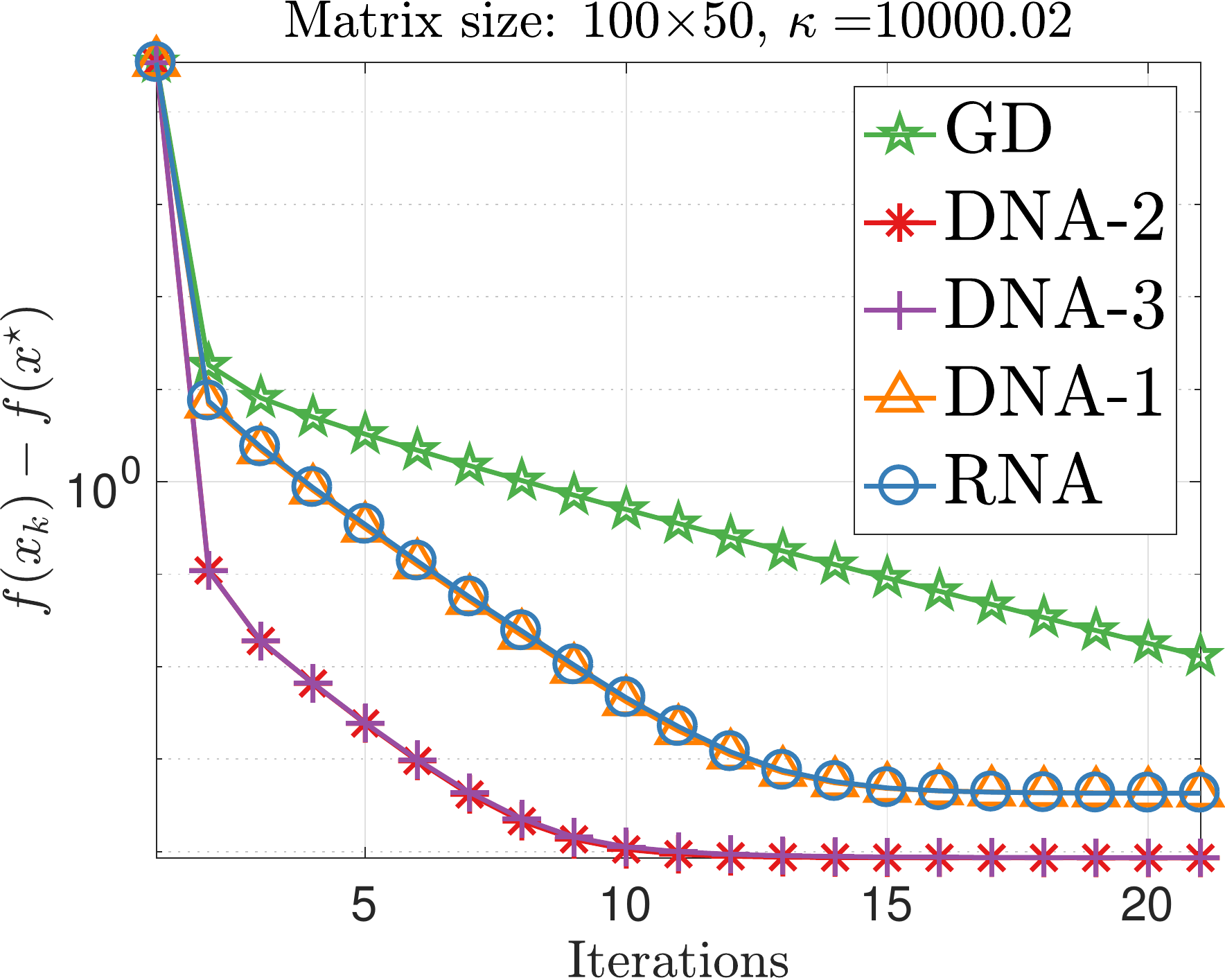}
	\end{minipage}
\begin{minipage}{0.3\textwidth}
		\centering
		\includegraphics[width=\textwidth]{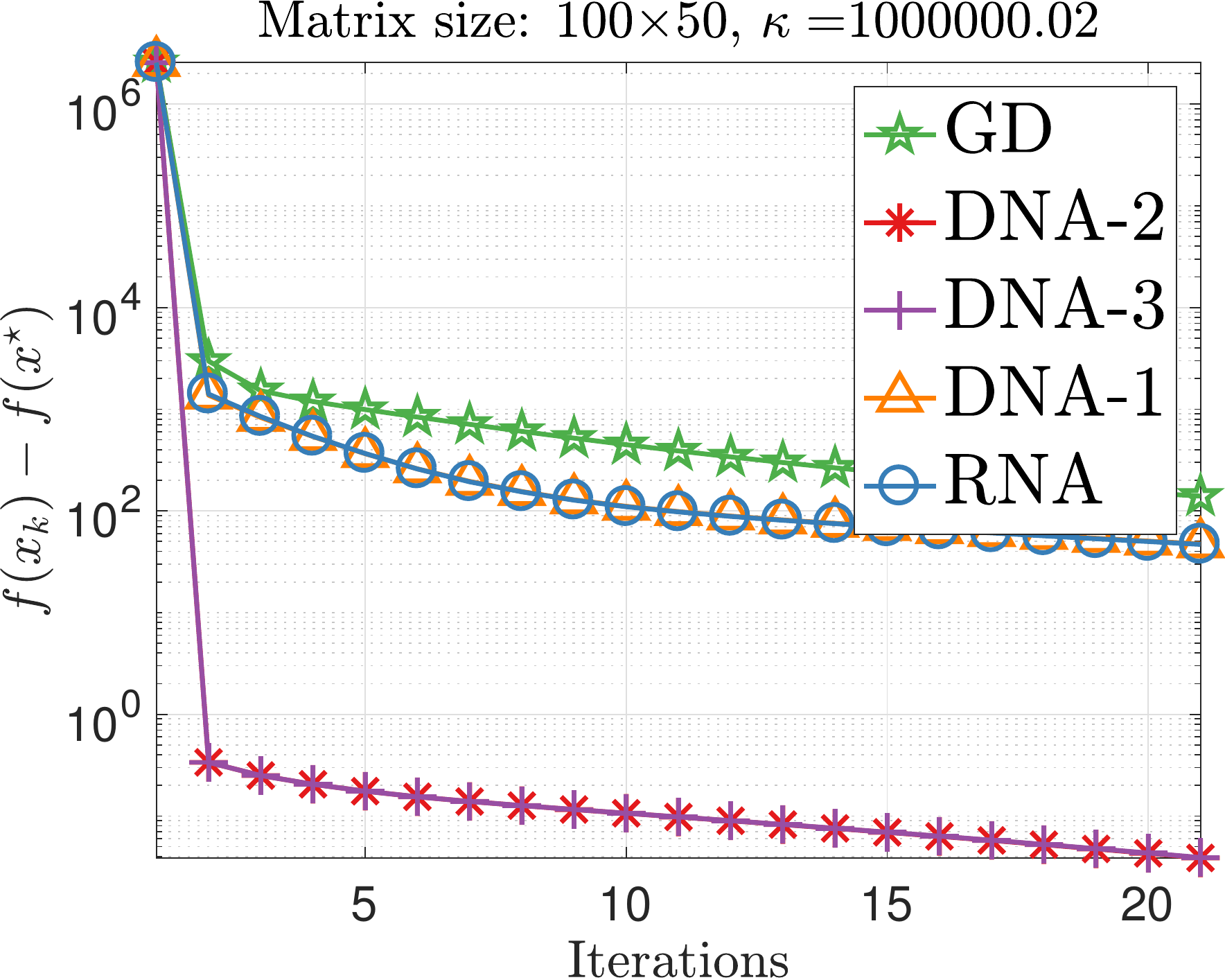}
	\end{minipage}
\begin{minipage}{0.3\textwidth}
		\centering
		\includegraphics[width=\textwidth]{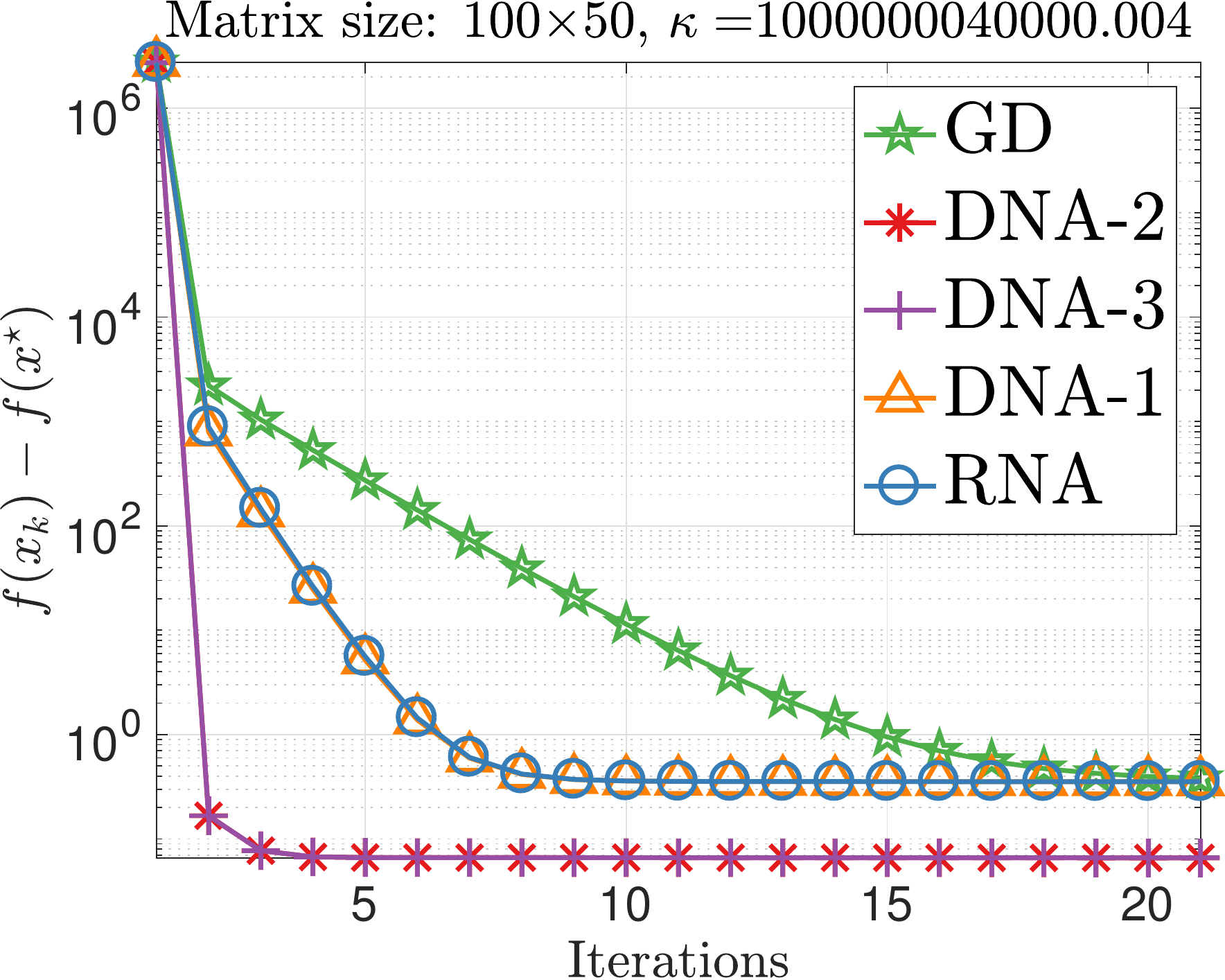}
	\end{minipage}
	\\
\begin{minipage}{0.3\textwidth}
		\centering
		\includegraphics[width=\textwidth]{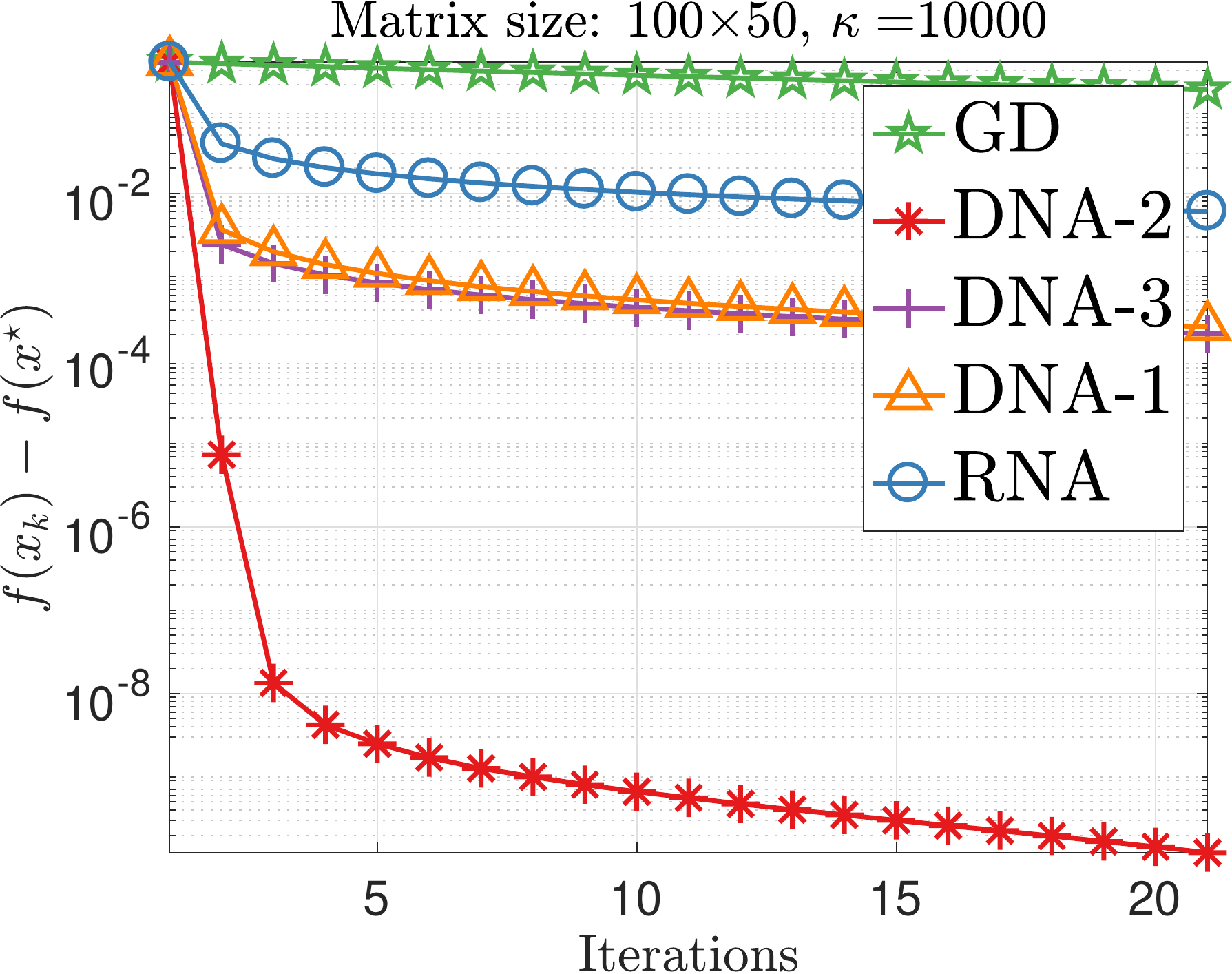}
	\end{minipage}
\begin{minipage}{0.3\textwidth}
		\centering
		\includegraphics[width=\textwidth]{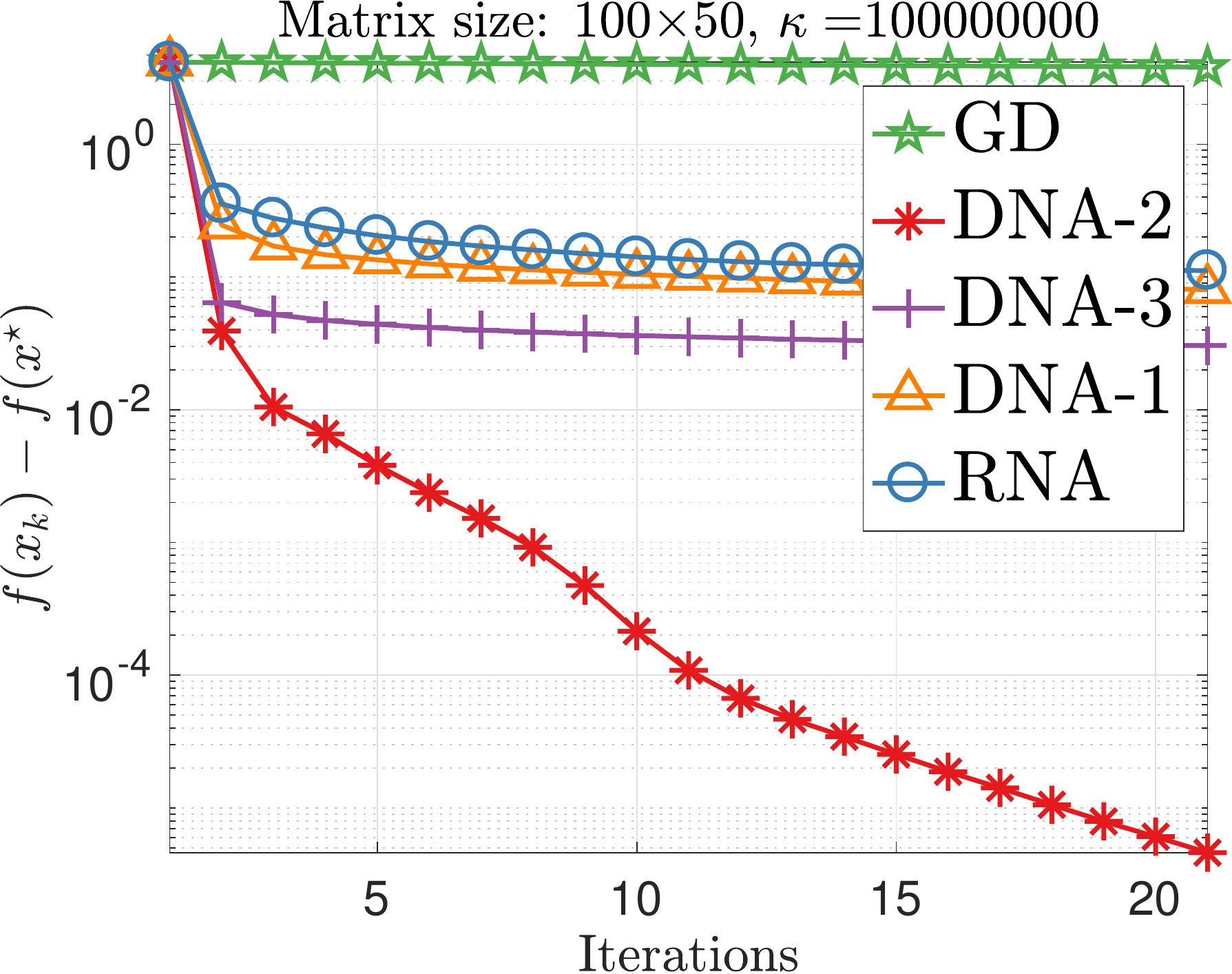}
	\end{minipage}
\begin{minipage}{0.3\textwidth}
		\centering
		\includegraphics[width=\textwidth]{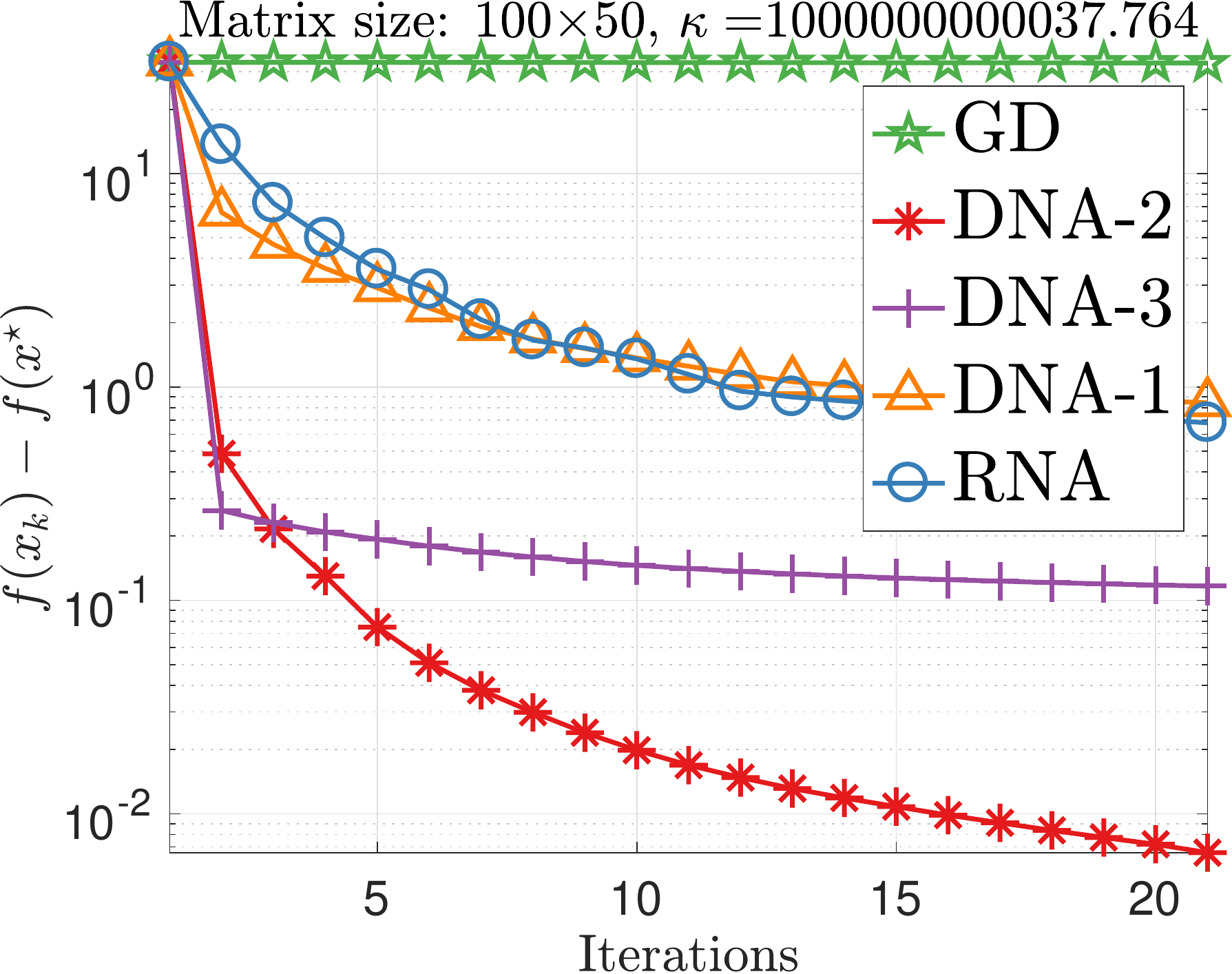}
	\end{minipage}
\caption{\small{Acceleration on synthetic data by using online acceleration scheme in \citep{rna_16}. First and second row represent quadratic, strong convex objective function as Least Squares and Ridge Regression, respectively. The last row represents non-quadratic but strong convex objective function as Logistic Regression. For all plots we use $k=3$. For RNA, we have $\lambda=10^{-8}$; for DNA, we set $\lambda=10^{-8}$, except for the last LR plot where for DNA-2, we set $\lambda=10$.} }\label{fig:synthetic_ol16}
\end{figure}

 \begin{algorithm}
	\SetAlgoLined
 	\SetKwInOut{Input}{Input}
	\SetKwInOut{Output}{Output}
     \SetKwInOut{Init}{Initialize}
     \SetKwInOut{Compute}{Compute}
\Input{Sequence of iterates $x_0,\ldots,x_{K+1}$ and sequence of step sizes $\alpha_0,\ldots,\alpha_{K}$\;}
	
		\nl Set $R =\left[\frac{x_0 - x_{1}}{\alpha_0} -\nabla f(0),\ldots,\frac{x_K - x_{K+1}}{\alpha_K} -\nabla f(0)  \right]$ and $X = [x_0, \ldots,x_K]$\;
		\nl Set $c$ as a solution of the linear system $X^\top R z = - X^\top \nabla f(0)$\;
		
	\Output{$x = \sum_{k = 0}^K c_k x_k$.}
 	\caption{DNA } \label{alg:fkRNA}
 \end{algorithm} 
 
\paragraph{Comments on the convergence of DNA.}
Let $H$ be the Hessian of $f$, where we assume that $f$  is quadratic.    
Also let $\lambda_{\rm max}(H)$ be the maximum eigenvalue of $H$ 
\begin{lemma}\label{lemma:krna}
 Let $c_D$ and $c_R$ be the extrapolation coefficients produced by DNA and RNA, respectively. Then 
$
\compactify{\|Xc_D-x^\star\|_H^2\le \|Xc_R-x^\star\|_H^2\le \lambda_{\rm max}(H)\|Xc_R-x^\star\|^2.}
$
\end{lemma}
\begin{theorem}\label{thm:dna_conv}
Denote  $\xi = \nicefrac{(\sqrt{L}-\sqrt{\mu})}{(\sqrt{L}+\sqrt{{\mu}})}.$ It is given in \citep{rna_16} that for the coefficients $c_R$ produced by Alg~\ref{alg:RNA}:
$
\compactify{\|Xc_R-x^\star\|^2\le \kappa(H)\tfrac{2\xi^k}{1+\xi^{2k}}\|x_0-x^\star\|^2. }
$
Further,  for Alg~\ref{alg:kRNA} we have
$
\compactify{\|Xc_D-x^\star\|_H^2\le \lambda_{\rm max}(H)\kappa(H)\tfrac{2\xi^k}{1+\xi^{2k}}\|x_0-x^\star\|^2.}
$
\end{theorem}
\begin{remark}
The convergence rate for DNA for quadratic functions in Theorem \ref{thm:dna_conv} is the same as that for Krylov subspace methods (for example, conjugate gradient algorithm) up to a multiplicative scalar. 
\end{remark}

However, numerically DNA is unstable like RNA without regularization. In fact, the matrix $X^\top R$ can be very ill-conditioned and can lead to large errors in computing $c^\star$. Moreover, we accumulate errors in approximating the gradient via linearization as  our approximation of the gradient is valid only in the neighborhood of the iterates $x_0,\ldots,x_K$. To solve these problems, we propose three regularized versions of DNA by using three different regularizers in the form of $g(c)$ and show that they work well in practice. But one can explore different forms of $g(c)$ as regularizer.~We explain them in the following sections. 
\begin{algorithm}
	\SetAlgoLined
 	\SetKwInOut{Input}{Input}
	\SetKwInOut{Output}{Output}
     \SetKwInOut{Init}{Initialize}
     \SetKwInOut{Compute}{Compute}
\Input{Sequence of iterates $x_0,\ldots,x_{K+1}$; sequence of step sizes $\alpha_0,\ldots,\alpha_{K}$; and $\mathbbm{1}\in\R^{K+1},$ a vector of all 1s\;}
	\nl Set $\tilde{R} =\left[\frac{x_0 - x_{1}}{\alpha_0},\ldots,\frac{x_K - x_{K+1}}{\alpha_K}  \right]$ and $X = [x_0, \ldots,x_K]$\;
		\nl Solve the linear system for $z\in\R^{K+1}$: $X^\top\tilde{R} z = \mathbbm{1}$\;
		\nl Set $c = \frac{z}{z^\top\mathbbm{1}}\in\mathbb{R}^{K+1}$\;	
	\Output{$x = \sum_{k = 0}^K c_k x_k$.}
 	\caption{
 	\kRNA -1} \label{alg:kRNA}
 \end{algorithm}

\vspace{-05pt}
 \subsection{DNA-1}
 \vspace{-05pt}

 This regularized version of \kRNA is directly influenced by Scieur et al. \citep{rna_16}. 
Here, we generate the extrapolated point $x$ as a linear combination of the set of $K+1$ iterates  such that, $x=\sum_k c_kx_k$. Additionally, as in \citep{rna_16,rna_18}, we assume the sum of the coefficients $c_k$ to be equal to 1.
Therefore, for $c\in\R^{K+1}$ with sum of its elements equal to 1, we set $g(c) =  1_{\sum_i c_i=1}$  in \eqref{mainDNA} and consider the following constrained problem:
 \begin{eqnarray}\label{eq:krna}
\compactify {{\boxed{\min_{c \in\R^{K+1}}\ f(Xc)+\lambda 1_{S}(c)=\min_{c} \left\{ f\left(Xc \right) \;:\; c\in\R^{K+1}, \; \compactify \sum_{k = 0}^K c_k =1 \right\},}}}
 \end{eqnarray}
where $X = [x_0, \ldots,x_K]$. We call this version of DNA as DNA-1. 
 \begin{lemma}\label{lemma:ckrna_lemma}
 If the objective function $f$ is quadratic and $X^\top \tilde{R}$ is nonsingular then 
  $c =  (X^\top \tilde{R})^{-1}\mathbbm{1}  / \delta, $ where $\delta=\mathbbm{1}^\top   (X^\top \tilde{R} )^{-1}\mathbbm{1},
  $
 and $\mathbbm{1}$ is the vector of dimension $K+1$ with all the components equal to 1 and   $ \compactify{\tilde{R}=\left[\nicefrac{(x_0 - x_{1})}{\alpha_0},\ldots,\nicefrac{(x_K - x_{K+1})}{\alpha_K}  \right]}$.
 \end{lemma}
 \begin{figure*}
\centering
	\begin{minipage}{0.24\textwidth}
		\centering
		\includegraphics[width=\textwidth]{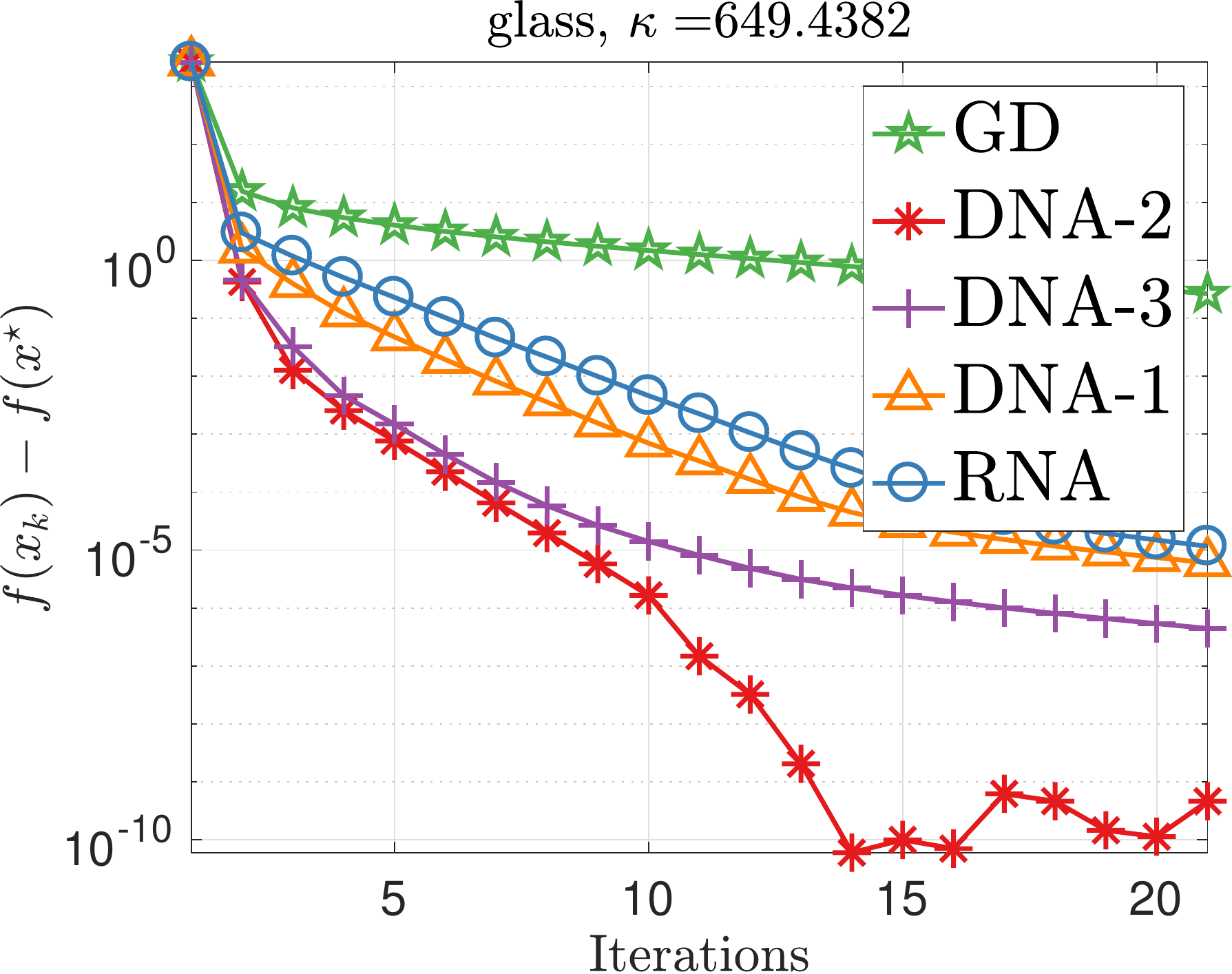}
		\end{minipage}
\begin{minipage}{0.24\textwidth}
		\centering
		\includegraphics[width=\textwidth]{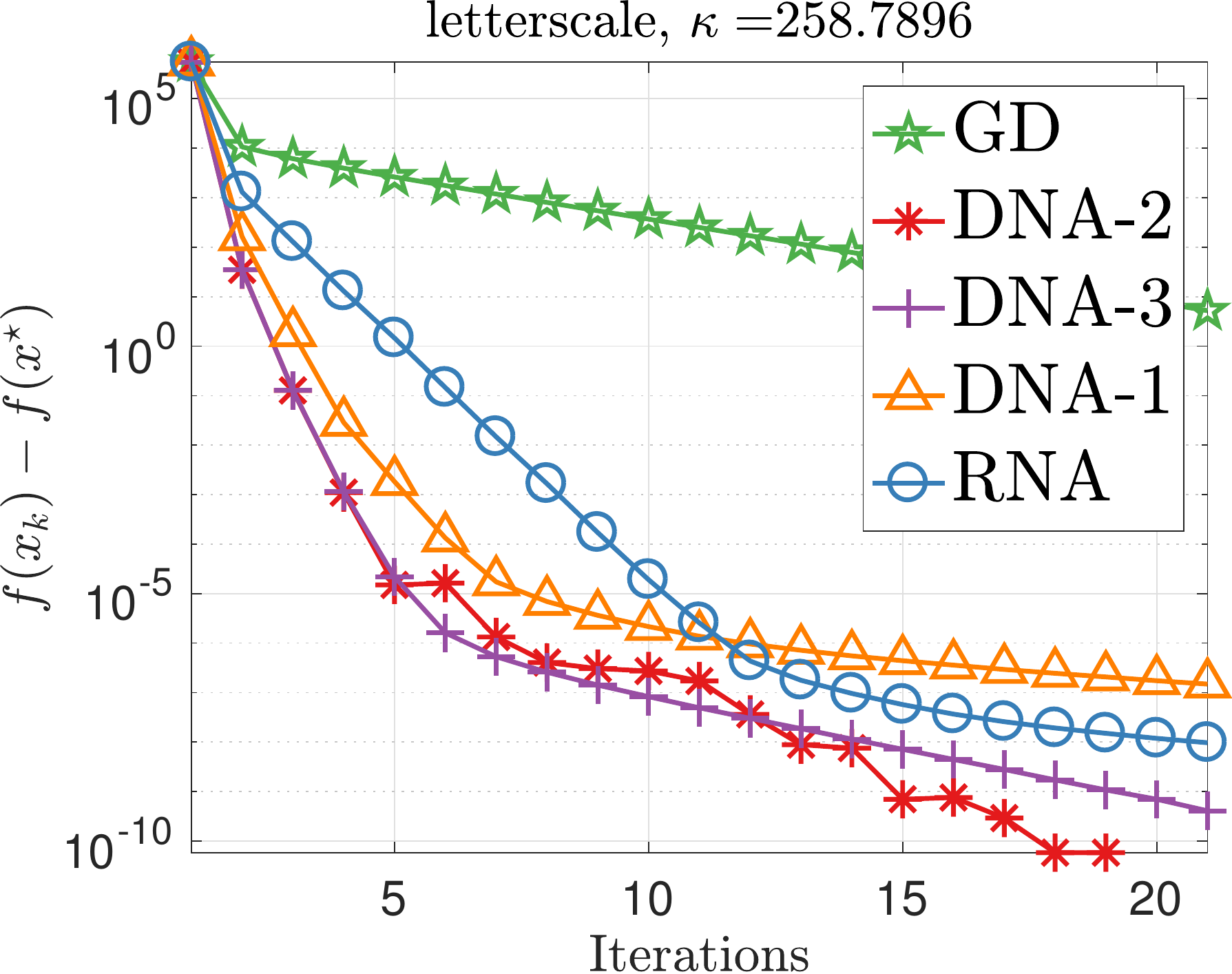}
		\end{minipage}
\begin{minipage}{0.24\textwidth}
		\centering
		\includegraphics[width=\textwidth]{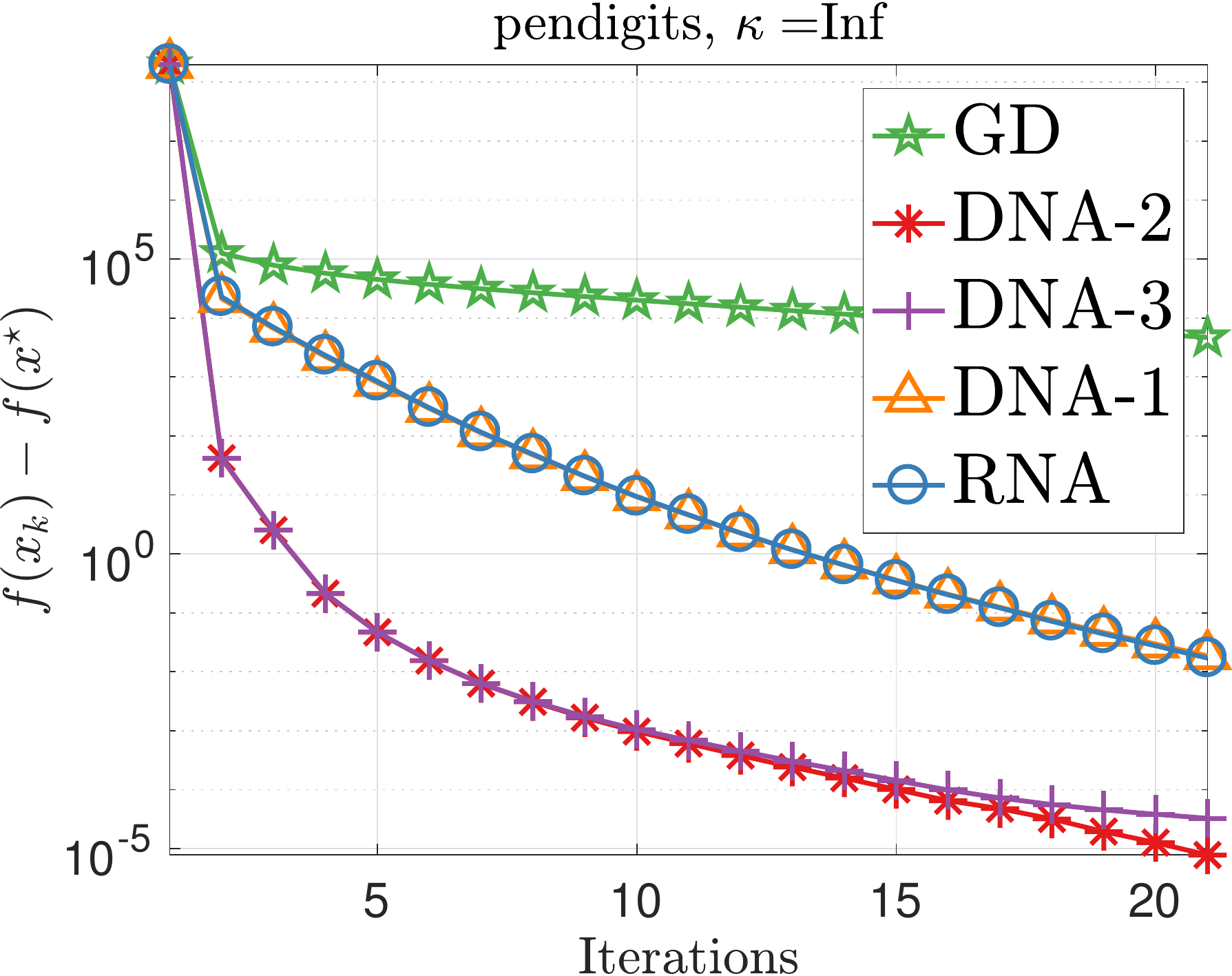}
		\end{minipage}
		\begin{minipage}{0.24\textwidth}
		\centering
		\includegraphics[width=\textwidth]{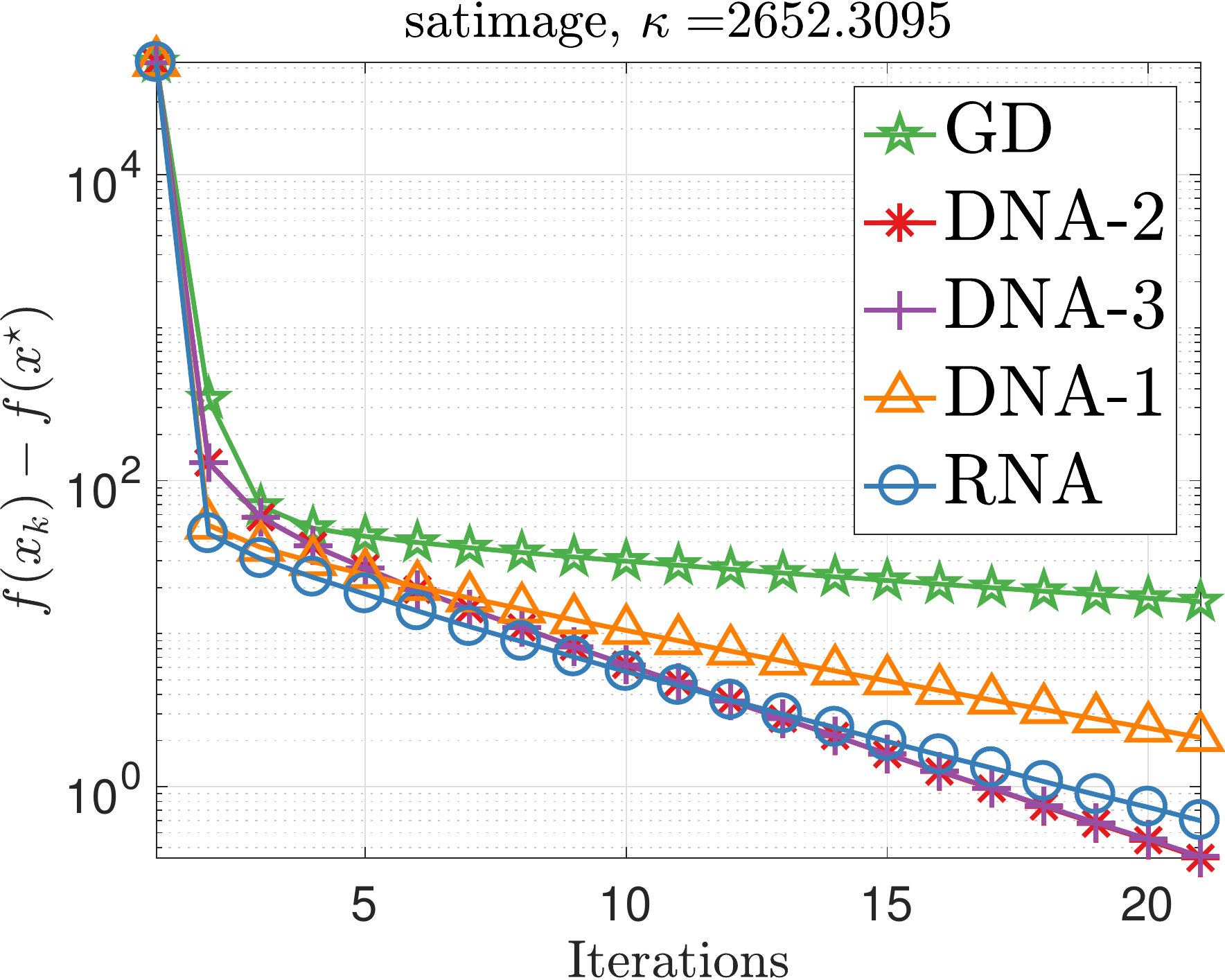}
		\end{minipage}
		\\
\begin{minipage}{0.24\textwidth}
		\centering
		\includegraphics[width=\textwidth]{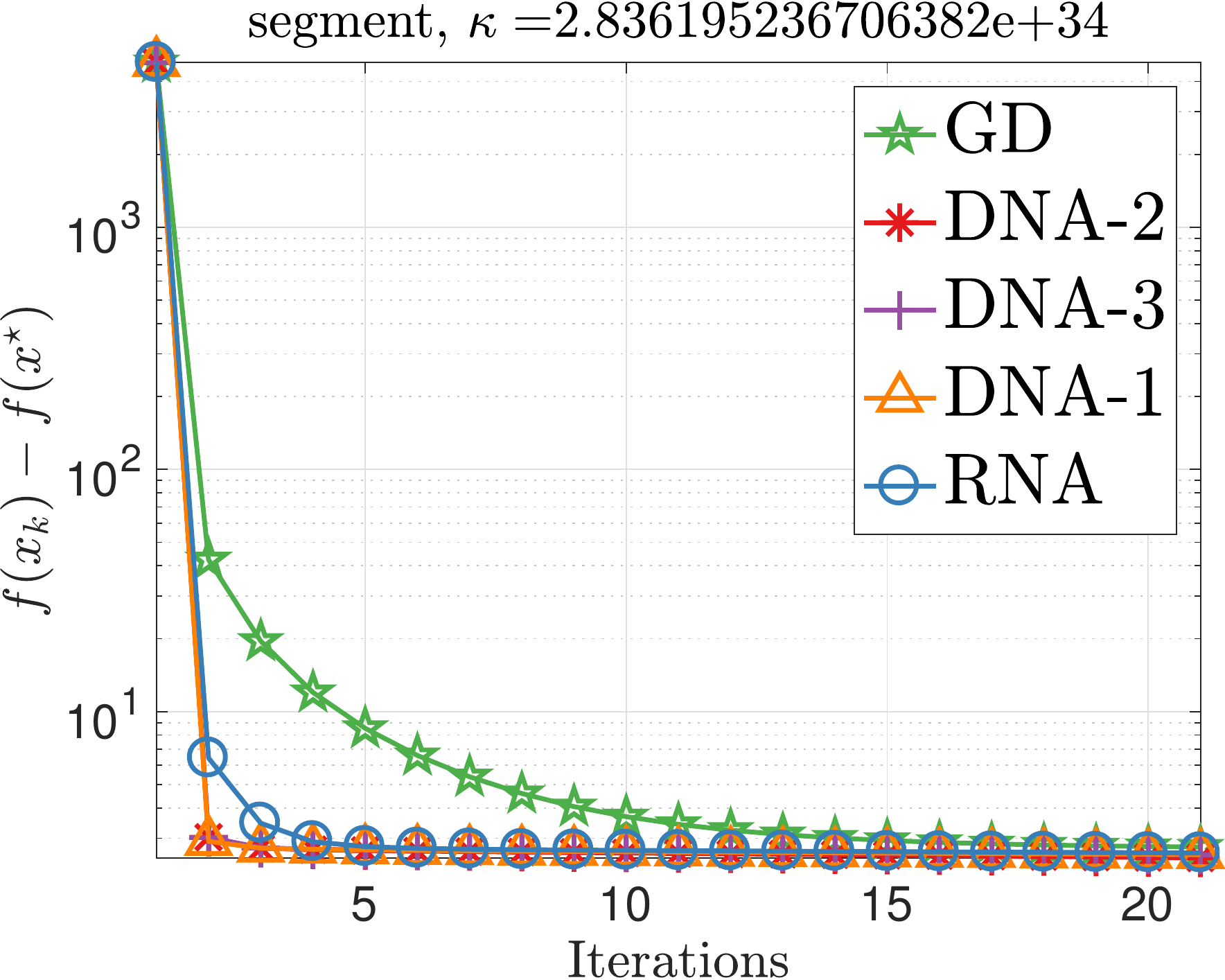}
		\end{minipage}
\begin{minipage}{0.24\textwidth}
		\centering
		\includegraphics[width=\textwidth]{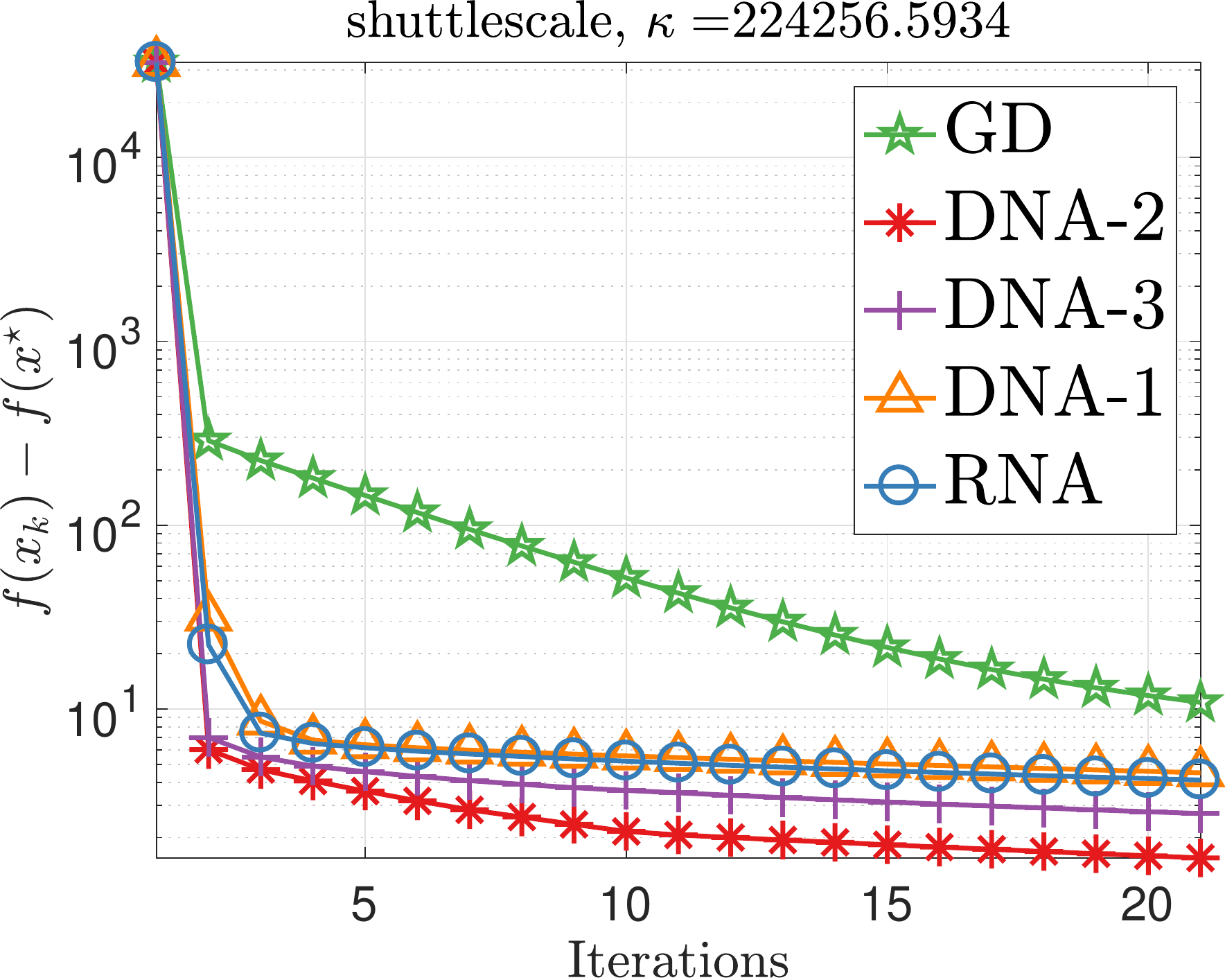}
		\end{minipage}
\begin{minipage}{0.24\textwidth}
		\centering
		\includegraphics[width=\textwidth]{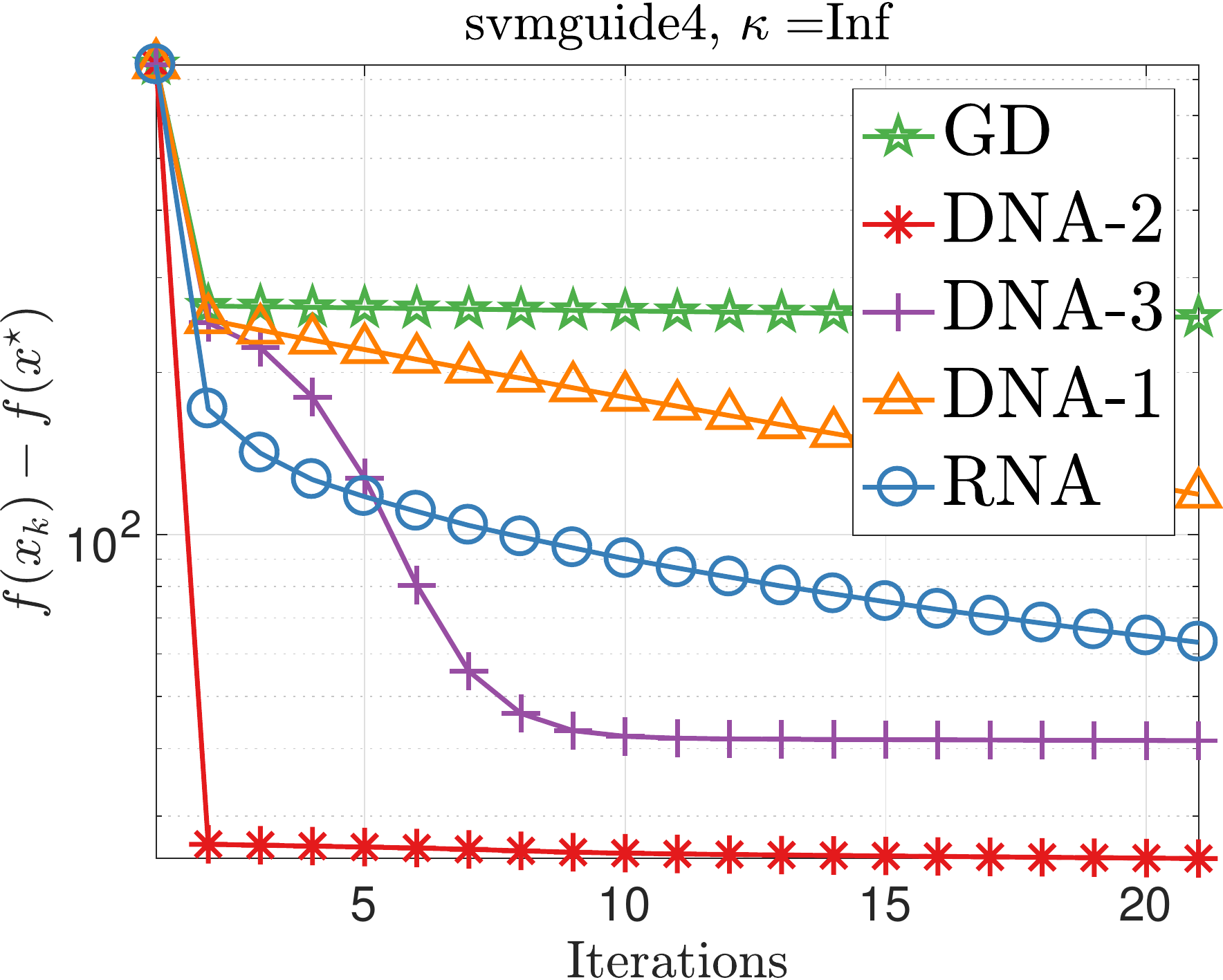}
		\end{minipage}
\begin{minipage}{0.24\textwidth}
		\centering
		\includegraphics[width=\textwidth]{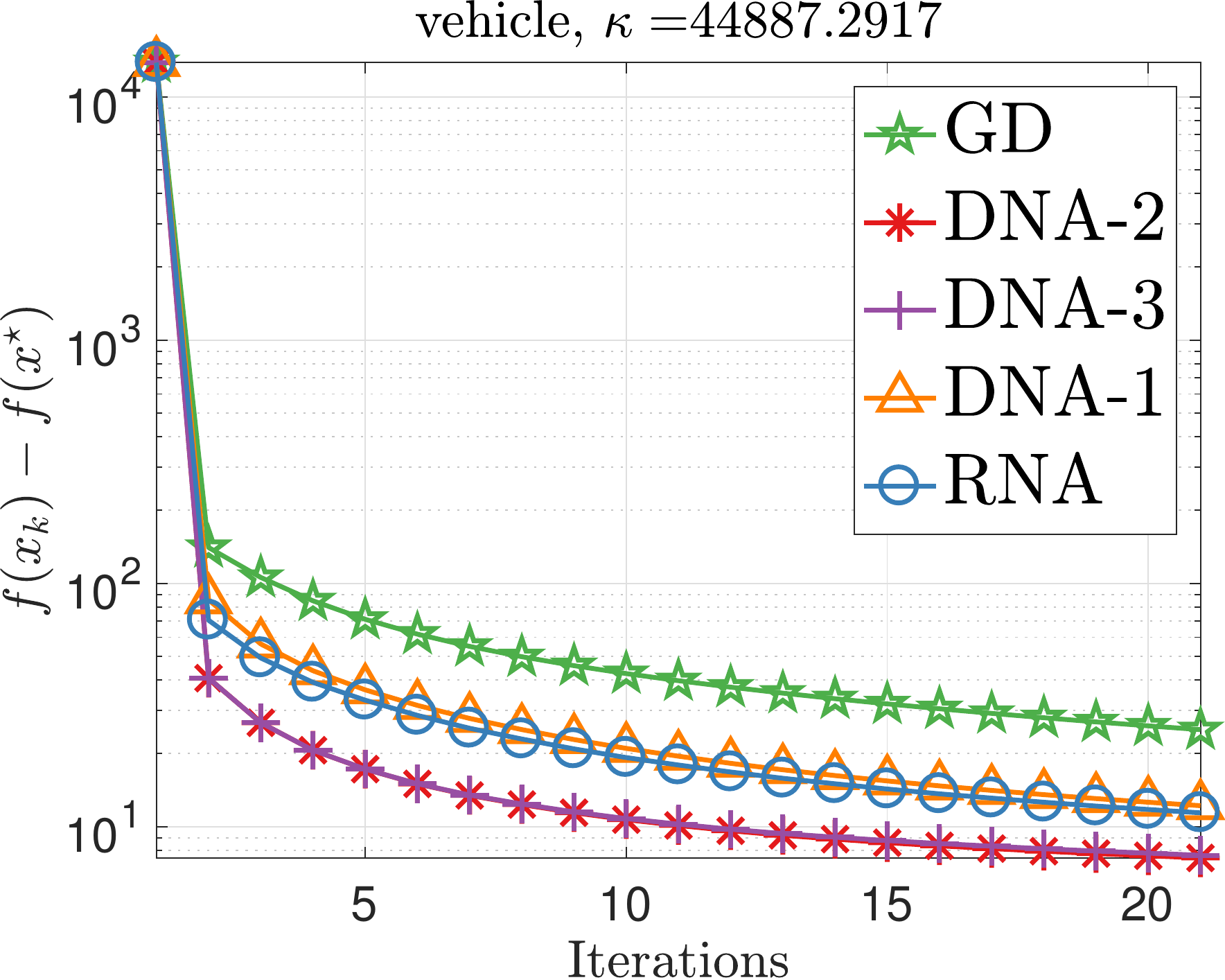}
		\end{minipage}
\caption{\small{Acceleration on {\tt LIBSVM} dataset by using online acceleration scheme in \citep{rna_16} on Least Squares problems. For all datasets, we use $k=3$. For RNA and DNA, we set $\lambda=10^{-8}$.}}\label{fig:real_ls_ol16}
\end{figure*}
\begin{figure*}
\centering
	\begin{minipage}{0.24\textwidth}
		\centering
		\includegraphics[width=\textwidth]{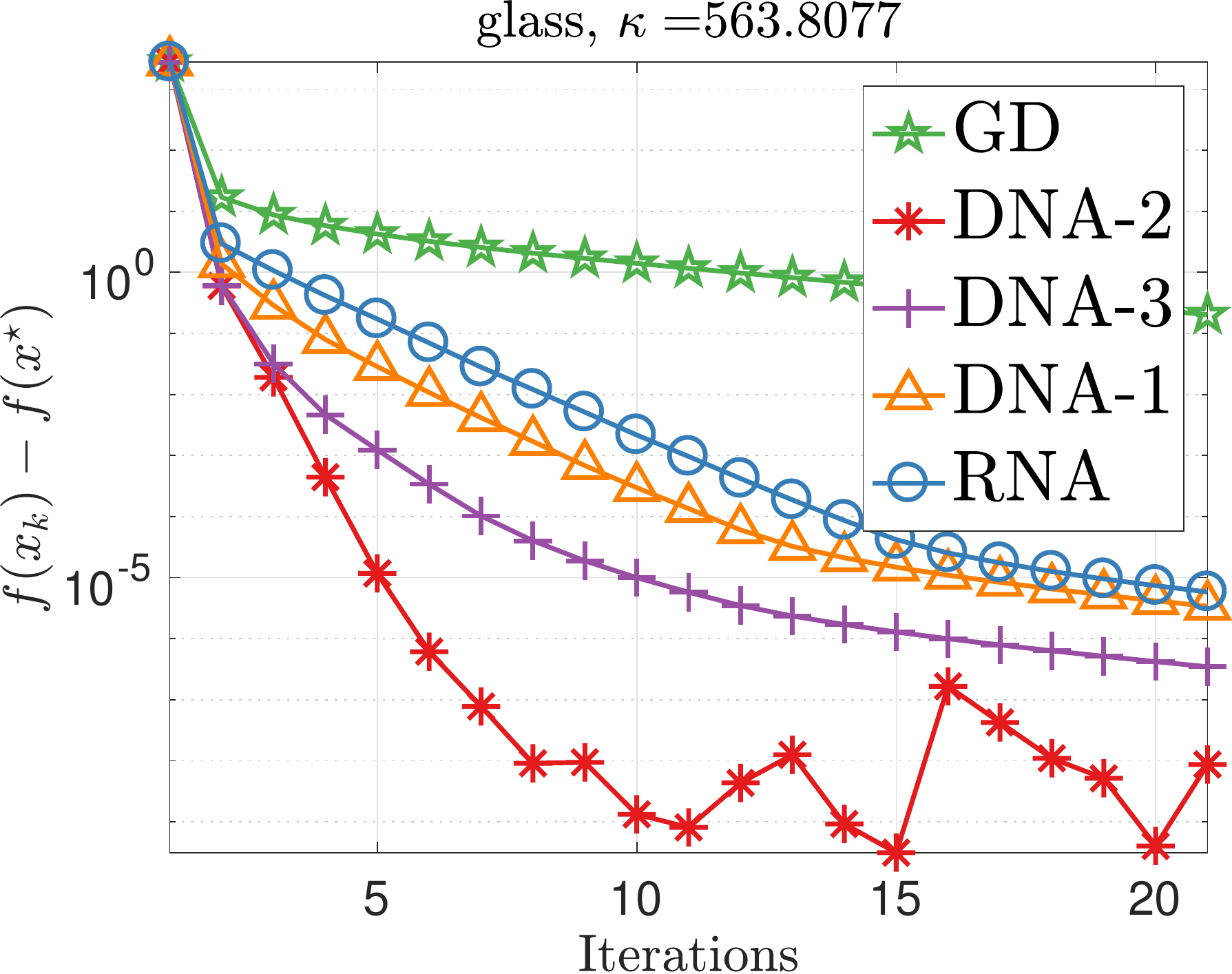}
		\end{minipage}
\begin{minipage}{0.24\textwidth}
		\centering
		\includegraphics[width=\textwidth]{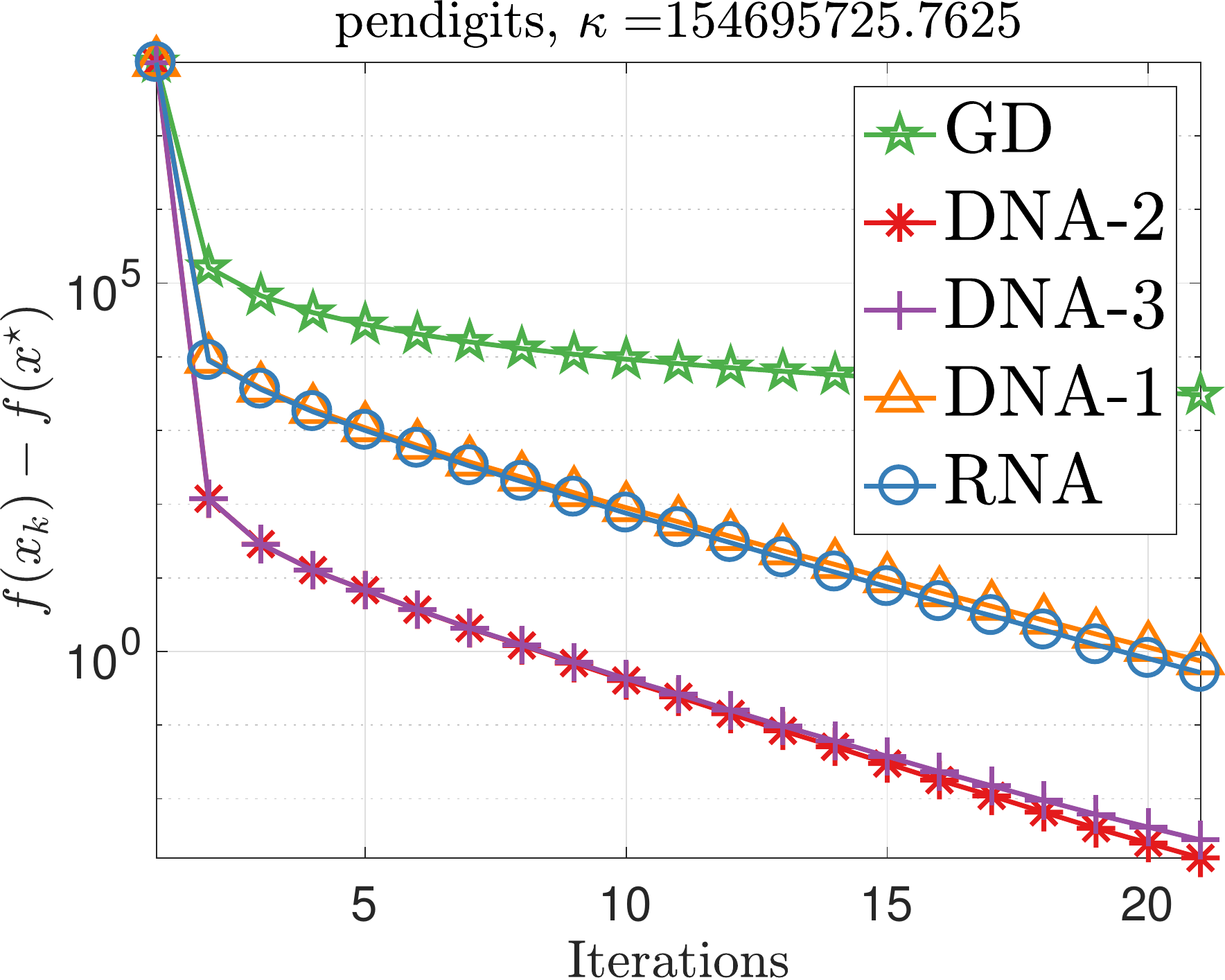}
		\end{minipage}
		\begin{minipage}{0.24\textwidth}
		\centering
		\includegraphics[width=\textwidth]{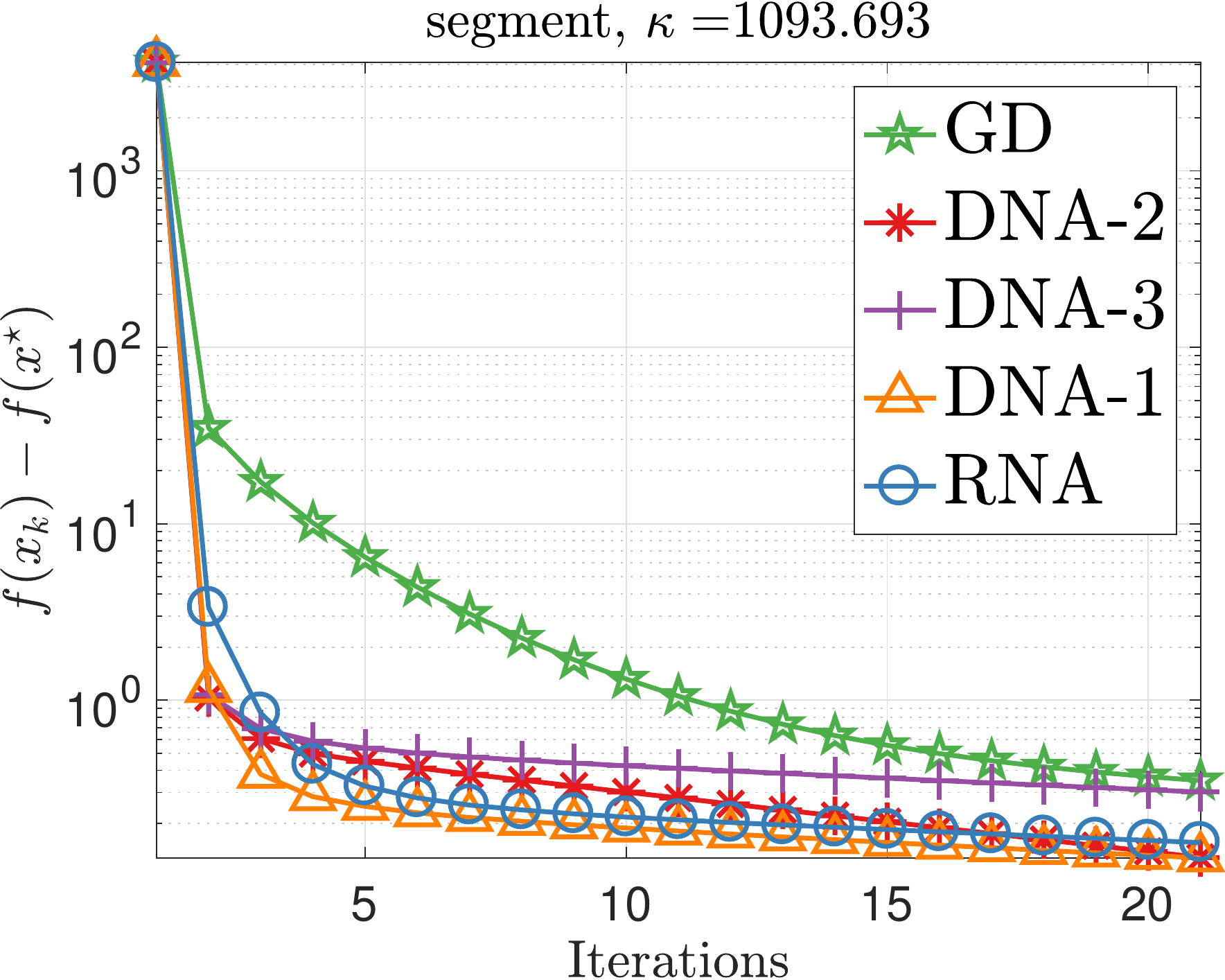}
			\end{minipage}		
\begin{minipage}{0.24\textwidth}
		\centering
		\includegraphics[width=\textwidth]{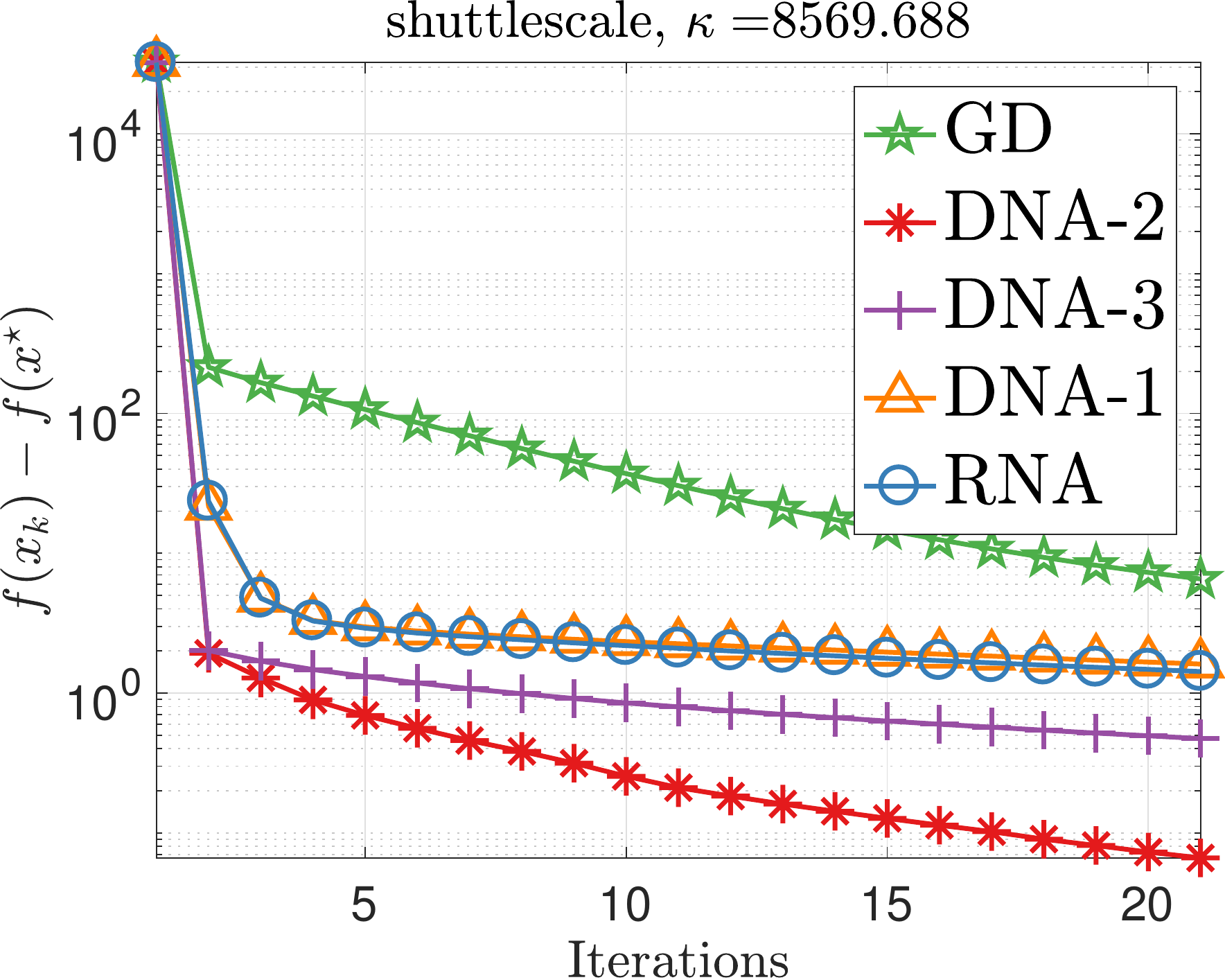}
		\end{minipage}
	\\
	\begin{minipage}{0.24\textwidth}
		\centering
		\includegraphics[width=\textwidth]{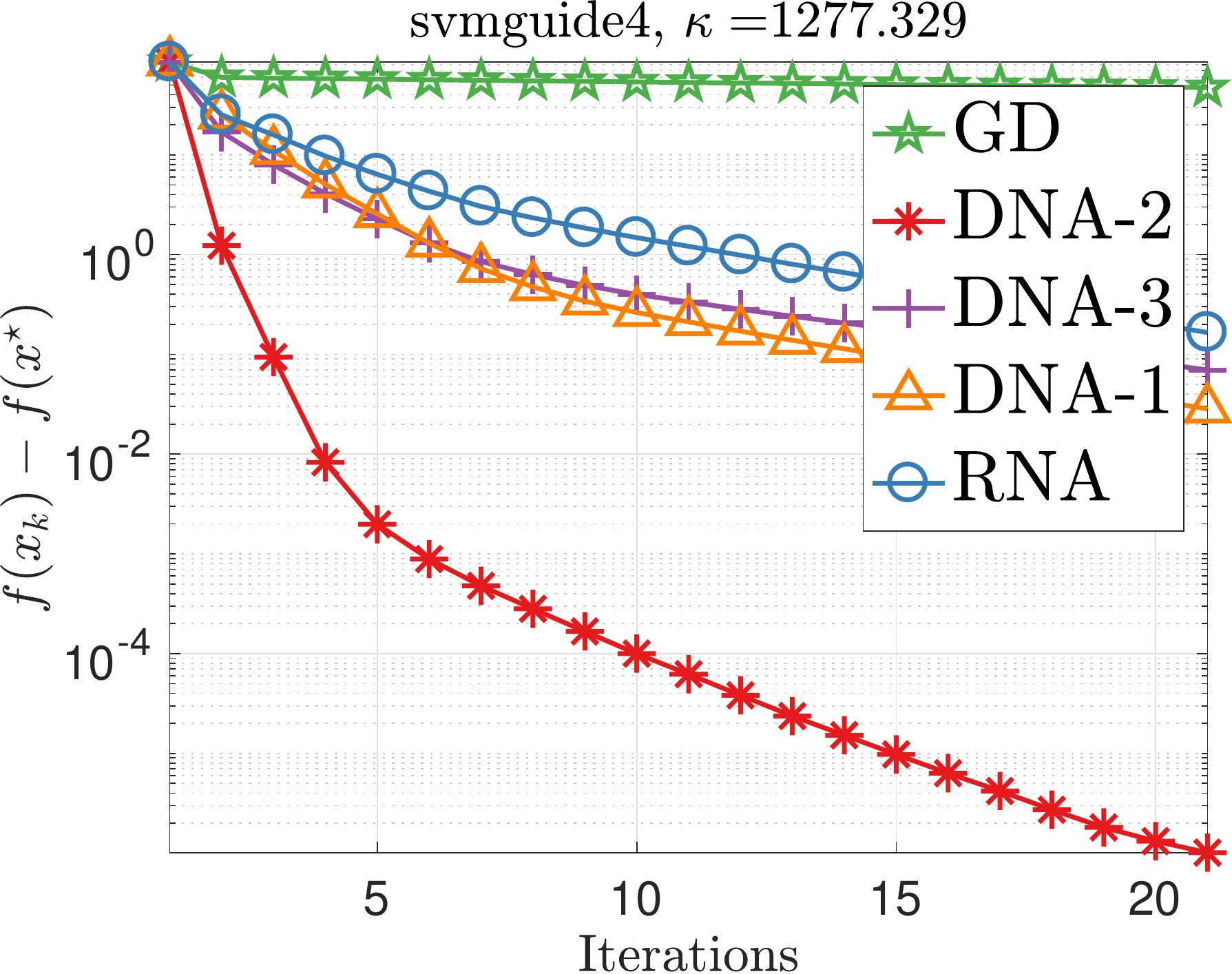}
	\end{minipage}
\begin{minipage}{0.24\textwidth}
		\centering
		\includegraphics[width=\textwidth]{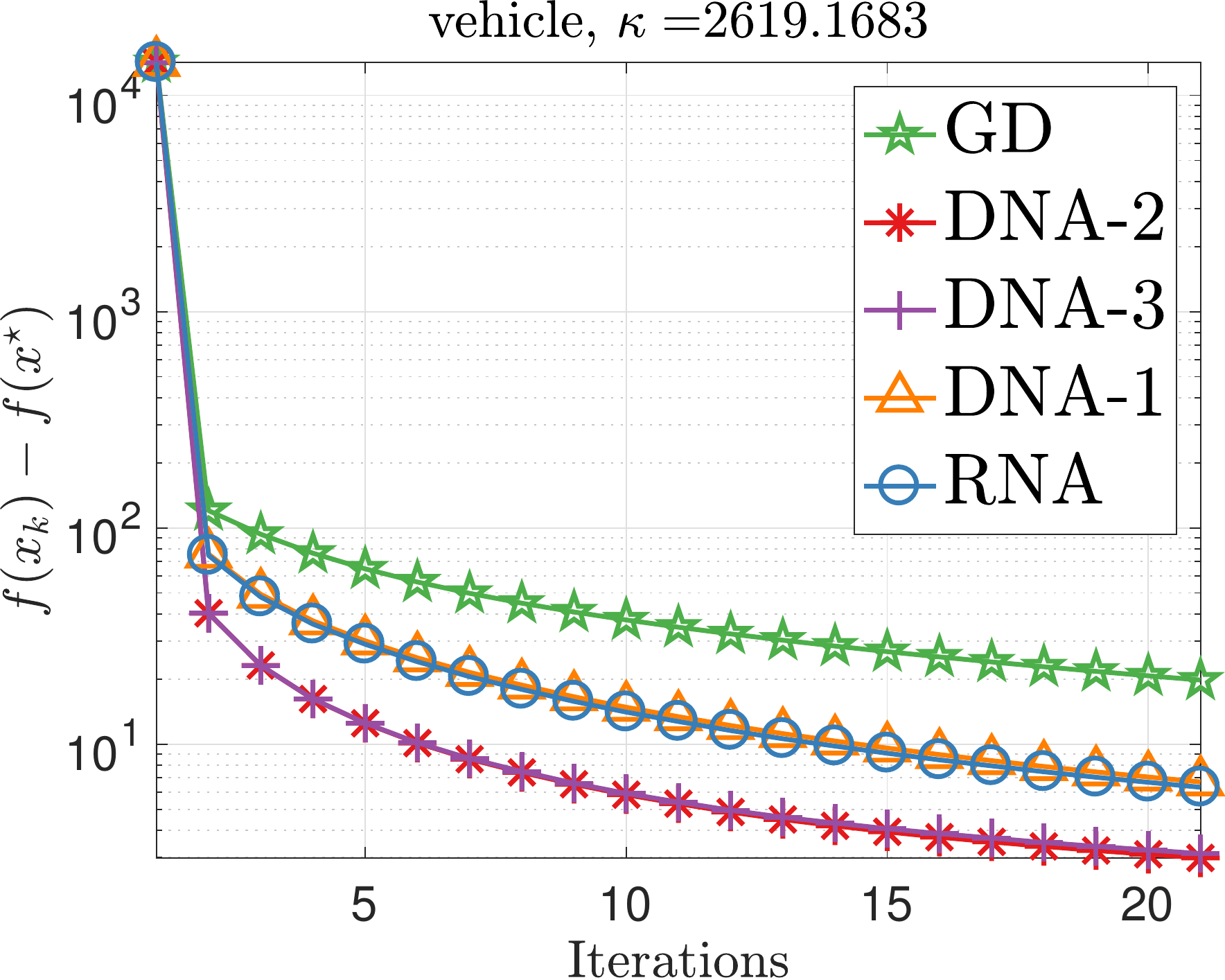}
		\end{minipage}
\begin{minipage}{0.24\textwidth}
		\centering
		\includegraphics[width=\textwidth]{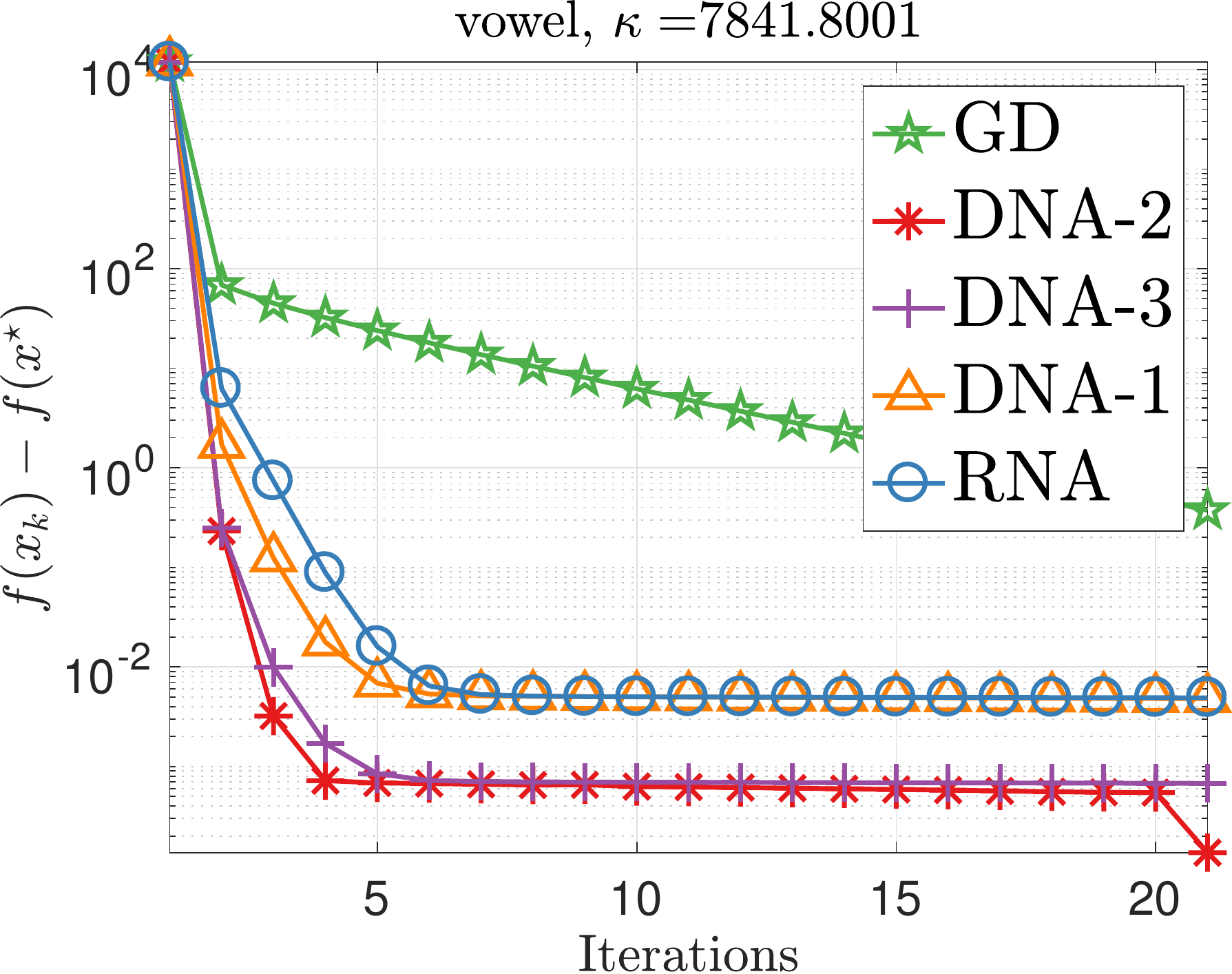}
		\end{minipage}
\begin{minipage}{0.24\textwidth}
		\centering
		\includegraphics[width=\textwidth]{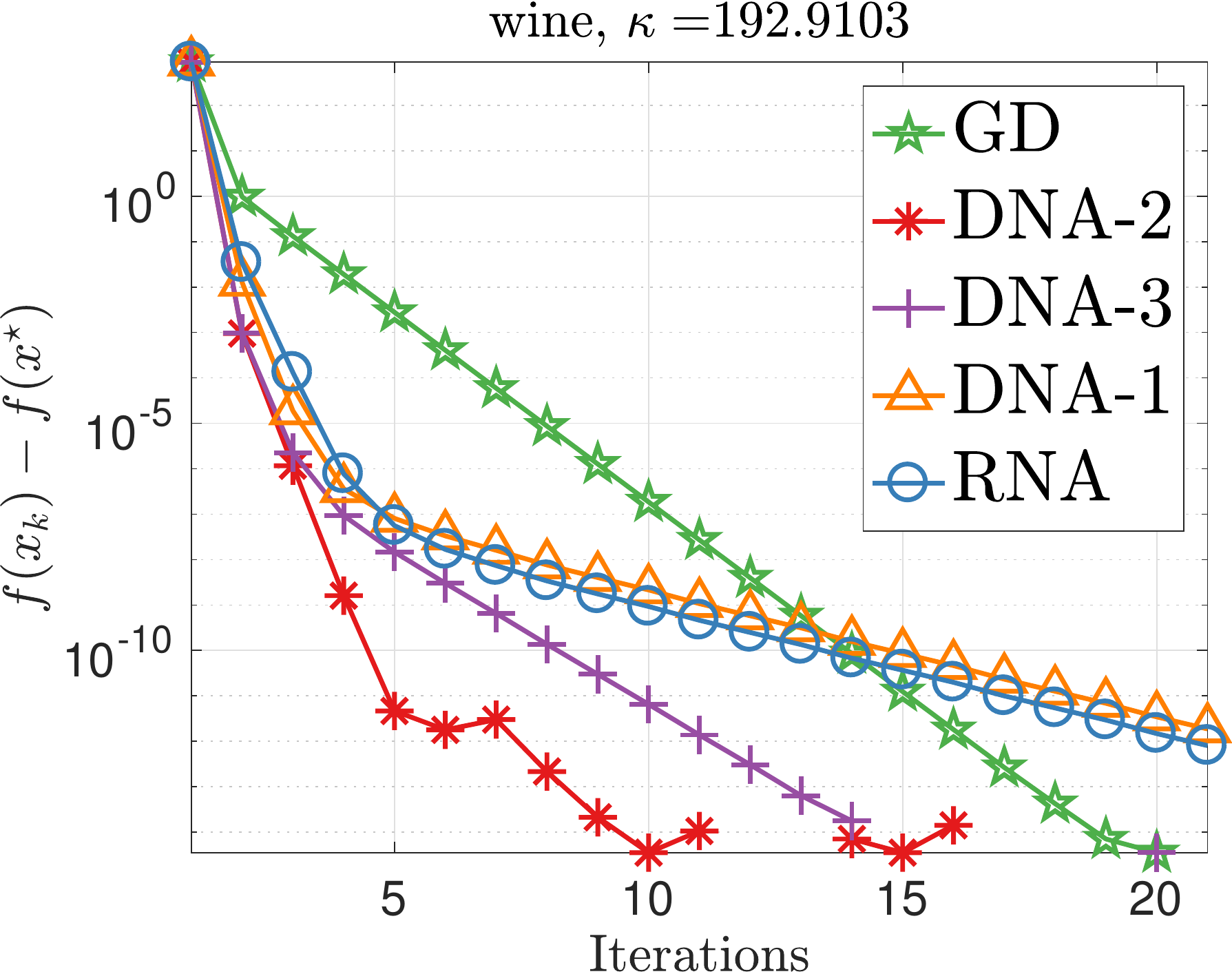}
		\end{minipage}
\caption{\small{Acceleration on {\tt LIBSVM} dataset by using online acceleration scheme in \citep{rna_16} on Ridge Regression problems. For all datasets, we use $k=3$. For RNA and DNA, we set $\lambda=10^{-8}$.}}\label{fig:realdata_rr_ol16}
\end{figure*}
 
Similar to RNA, if $X^\top \tilde{R}$ is singular then $c$ is not necessarily unique. Any $c$ of the form $\frac{z}{z^\top\mathbbm{1}}$, where $z$ is a solution of $X^\top\tilde{R}z = \mathbbm{1}$, is a solution 
 of (\ref{eq:krna}). DNA-1 is described in Alg~\ref{alg:kRNA}.
 \paragraph{Comparison with RNA on simple quadratic functions.}
 Denote the functional value obtained by DNA,  DNA-1 (Alg~\ref{alg:kRNA}) and RNA (Alg~\ref{alg:RNA}) at an extrapolated point as $f_{D}$, $f_{D1}$ and $f_R$, respectively. 
 \begin{proposition}\label{fn_value}
Let $A\in\R^{n\times n}$ be symmetric and positive definite and $f(x)=\frac{1}{2}x^\top Ax$ be a quadratic
objective function. Let $X=[x_0\;\;x_1\;\cdots x_k]$ be a matrix generated by stacking $k$ iterates of GD to minimize $f$. 
Then the functional value of DNA, DNA-1 and RNA at the accelerated point are: $f_D = 0$, 
$
\compactify{f_{\rm D1}=\frac{1}{2 \mathbbm{1}^\top (X^\top A X)^{-1} \mathbbm{1}}},
$
and  
$
\compactify{ f_{\rm R}=\frac{\mathbbm{1}^\top (X^\top A^2 X)^{-1}X^\top A X (X^\top A^2 X)^{-1}\mathbbm{1}}{2(\mathbbm{1}^\top (X^\top A^2X)^{-1} \mathbbm{1})^2}},
$
respectively. 
\end{proposition}
We conclude that for this simple objective function, DNA reaches the optimal solution after the first acceleration. Moreover, one can choose the matrix $A$ such that $f_R$ is arbitrary large, and this example shows that DNA may outperform RNA by a large margin.~The comparison between DNA-1 and RNA on the previous example is given in the following lemma and theorem. 
\begin{figure*}
\centering
	\begin{minipage}{0.26\textwidth}
		\centering
		\includegraphics[width=\textwidth]{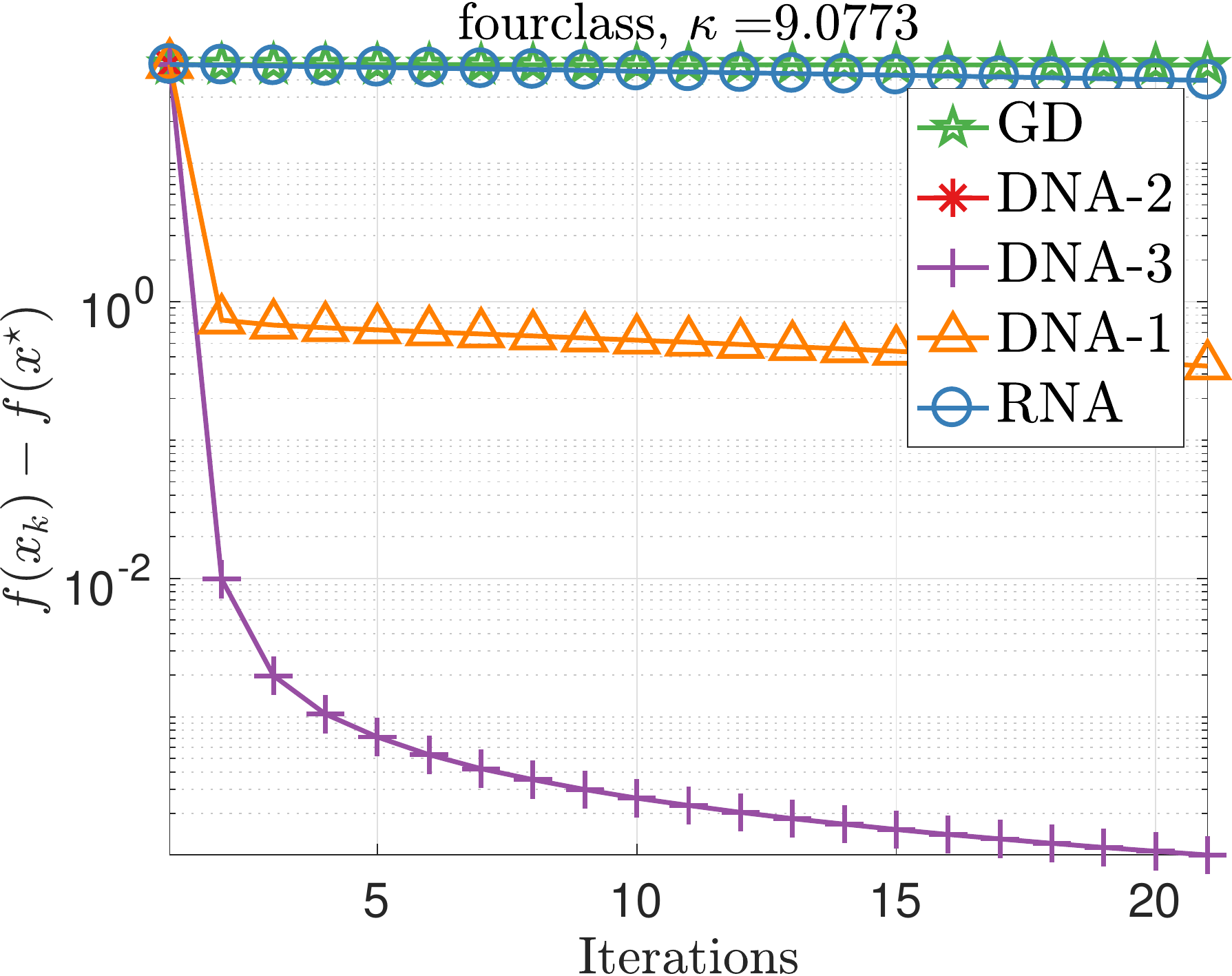}
		\end{minipage}
\begin{minipage}{0.26\textwidth}
		\centering
		\includegraphics[width=\textwidth]{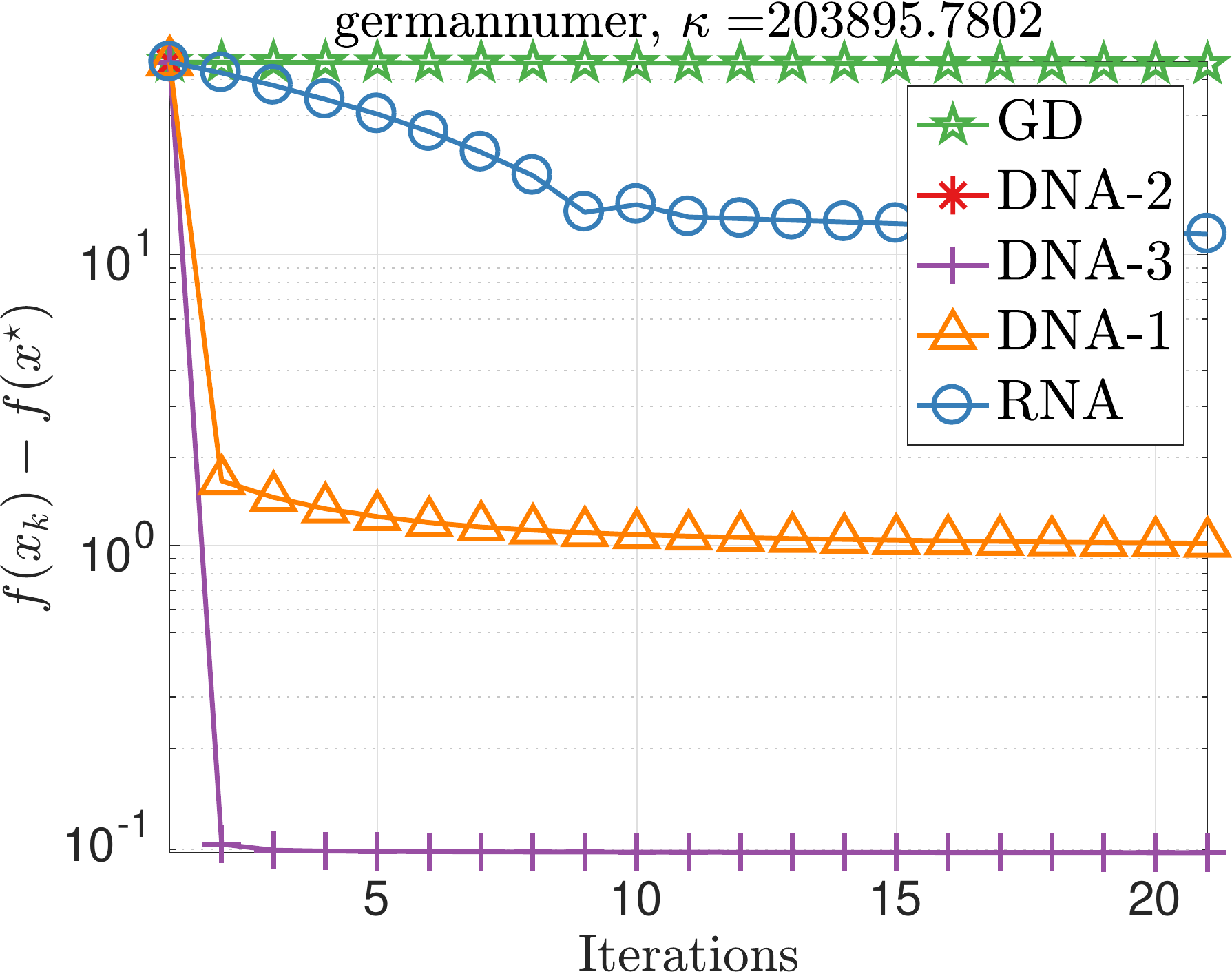}
		\end{minipage}
\begin{minipage}{0.26\textwidth}
		\centering
		\includegraphics[width=\textwidth]{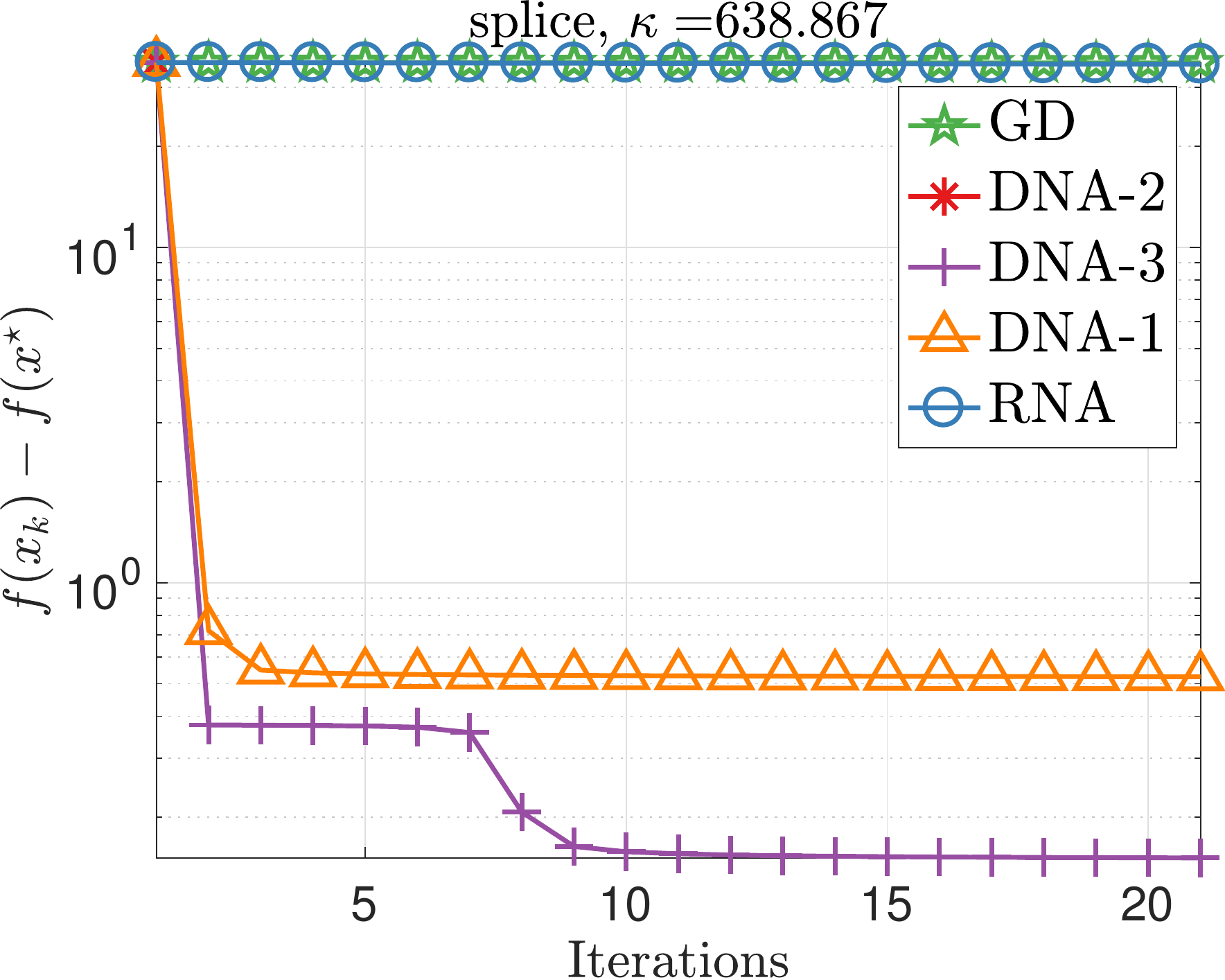}
		\end{minipage}
\caption{\small{Acceleration on {\tt LIBSVM} dataset by using online acceleration scheme in \citep{rna_16} on Logistic Regression problems. For all datasets, we use $k=3$. For RNA and DNA, we set $\lambda=10^{-8}$.}}\label{fig:realdata_lr_ol16}
\end{figure*}
 \begin{lemma}\label{lemma:ratio}
We assume that the matrix $\tilde{R}$ 
 has full column rank. 
With the notations used in Proposition \ref{fn_value}, we have $\compactify{\nicefrac{f_{{\rm R}}}{f_{{\rm D1}}}=\|z\|_{A^{-1}}^2\|y\|^2\|z\|^{-4},}$ where $\compactify{z\eqdef (\tilde{R}^\dagger)^\top \mathbbm{1} = ((AX)^\dagger)^\top \mathbbm{1}}$ and $\compactify{y\eqdef((A^{1/2}X)^\dagger)^\top\mathbbm{1}}$. 
 We have $y^\top A^{-1/2} z = z^\top z$; then, by using Cauchy-Schwarz inequality, we conclude that $\compactify{{\|z\|_{A^{-1}}\|y\|_2} \ge y^\top A^{-1/2} z= {\|z\|_2^2}}$
whence $\nicefrac{f_{\rm R}}{f_{\rm D1}} \ge 1$.
\end{lemma} 
Note that the ratio $\compactify{\frac{f_{\rm R}}{f_{\rm D1}} \ge 1}$ can be directly concluded from the definition of ${f_{\rm R}}$ and ${f_{\rm D1}}$. The main goal of the previous lemma is to exactly quantify the ratio between these two quantities. The following theorem gives more insight.
\begin{theorem}\label{lemma:UB}
We have $ \compactify{\frac{f_{{\rm R}}}{f_{{\rm D1}}}\le U_R \eqdef \|z\|_{A^{-1}}^2\|z\|_A^2 \|z\|^{-4},}$ and $U_R \in \left[\nicefrac{1}{2}+ \nicefrac{\kappa(A)}{2},\kappa(A)\right],$ where 
$\kappa(A)$ is the condition number of $A$. 
\end{theorem}
 \begin{algorithm}
	\SetAlgoLined
 	\SetKwInOut{Input}{Input}
	\SetKwInOut{Output}{Output}
     \SetKwInOut{Init}{Initialize}
     \SetKwInOut{Compute}{Compute}
\Input{Sequence of iterates $x_0,\ldots,x_{K+1}$; sequence of step sizes $\alpha_0,\ldots,\alpha_{K}$; regularizer $\lambda>0$; and reference vector $y\in\R^{k+1}$\;}
	
		\nl Set $R =\left[\frac{x_0 - x_{1}}{\alpha_0} -\nabla f(0),\ldots,\frac{x_K - x_{K+1}}{\alpha_K} -\nabla f(0)  \right]$ and $X = [x_0, \ldots,x_K]$\;
		\nl Set $c$ as a solution of the  linear system
		$(X^\top R+\lambda X^\top X) z= \lambda X^\top y-X^\top \nabla f(0)$\;
	\Output{$x = \sum_{k = 0}^K c_k x_k$.}
 	\caption{ DNA-2 } \label{alg:DNA-2}
 \end{algorithm} 

The above theorem tells us, for a simple quadratic function, the ratio of the objective function values of DNA-1 and RNA may attain an order of $\kappa(A)$, but it never exceeds it. The theoretical quantification of the acceleration obtained by DNA and its different versions compared to RNA in more general problems is left for future work.   
Although DNA-1 can be seen as a regularized version of DNA, we still need to remedy the fact that the linearization of the gradient is not a {\em good} approximation in the entire space, and that the matrix $X^\top\tilde{R}$ may be singular. To this end, we impose some regularization such that the new extrapolated point {\em stays} near to some reference point. We propose two different ways in the following two sections.  
\subsection{DNA-2} 
We set  $g(c)=\|Xc-y\|^2$ in \eqref{mainDNA} and consider a regularized version of problem \eqref{eq:fkrna}:
\begin{eqnarray}\label{eq:fkrnareg-}
\compactify \min_{\substack{c\in\R^{K+1}}}  f\left(Xc \right)+\frac{\lambda}{2}\|Xc-y\|^2,
 \end{eqnarray}
 where $\lambda>0$ is a balancing parameter and $y$ is a reference point (a point supposed to be in the neighborhood of $Xc$). 
 By taking the derivative of the objective in \eqref{eq:fkrnareg-} with respect to $c$ and setting it to 0, we find 
$
\compactify{ X^\top\nabla f(Xc)+\lambda X^\top(Xc-y)=0},
$
which after using the approximation \eqref{approximate_grad} becomes
$
\compactify{ X^\top(Rc+\nabla f(0) )+\lambda X^\top(Xc-y)=0.}
$
Finally, $c^\star$ is given as a solution to the linear system 
\begin{eqnarray}\label{solution2}
\compactify (X^\top R+\lambda X^\top X) c=\lambda X^\top y-X^\top \nabla f(0).
\end{eqnarray}
In general, $X^\top R$ is not necessarily symmetric. To justify the regularization further, one might symmetrize $X^\top R$ by its transpose. In our experiments, we obtained good performance without this. 
\begin{figure*}
\centering
	\begin{subfigure}{0.3\textwidth}
		\centering
		\includegraphics[width=\textwidth]{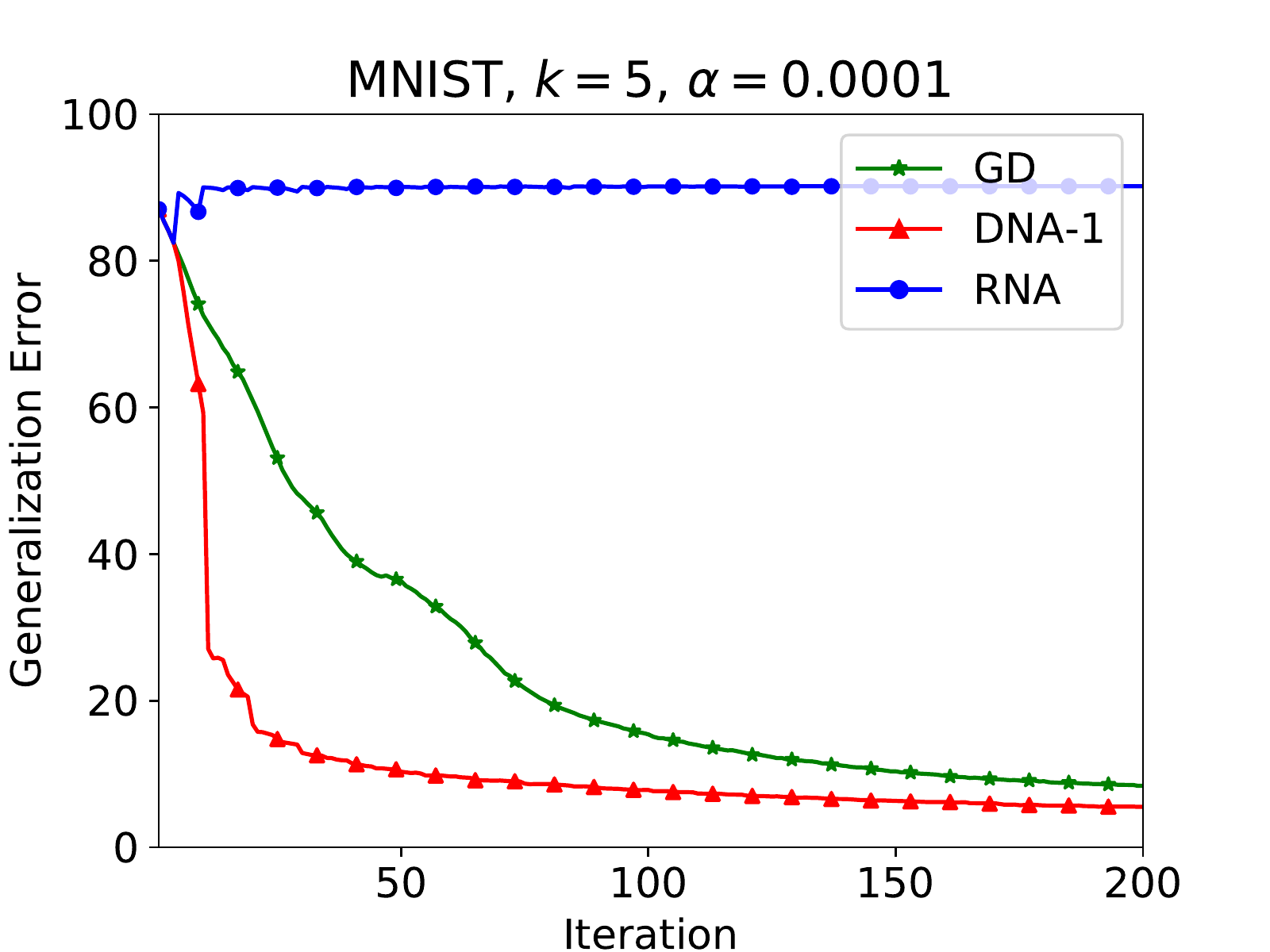}
				\caption{}
\end{subfigure}
\begin{subfigure}{0.3\textwidth}
		\centering
		\includegraphics[width=\textwidth]{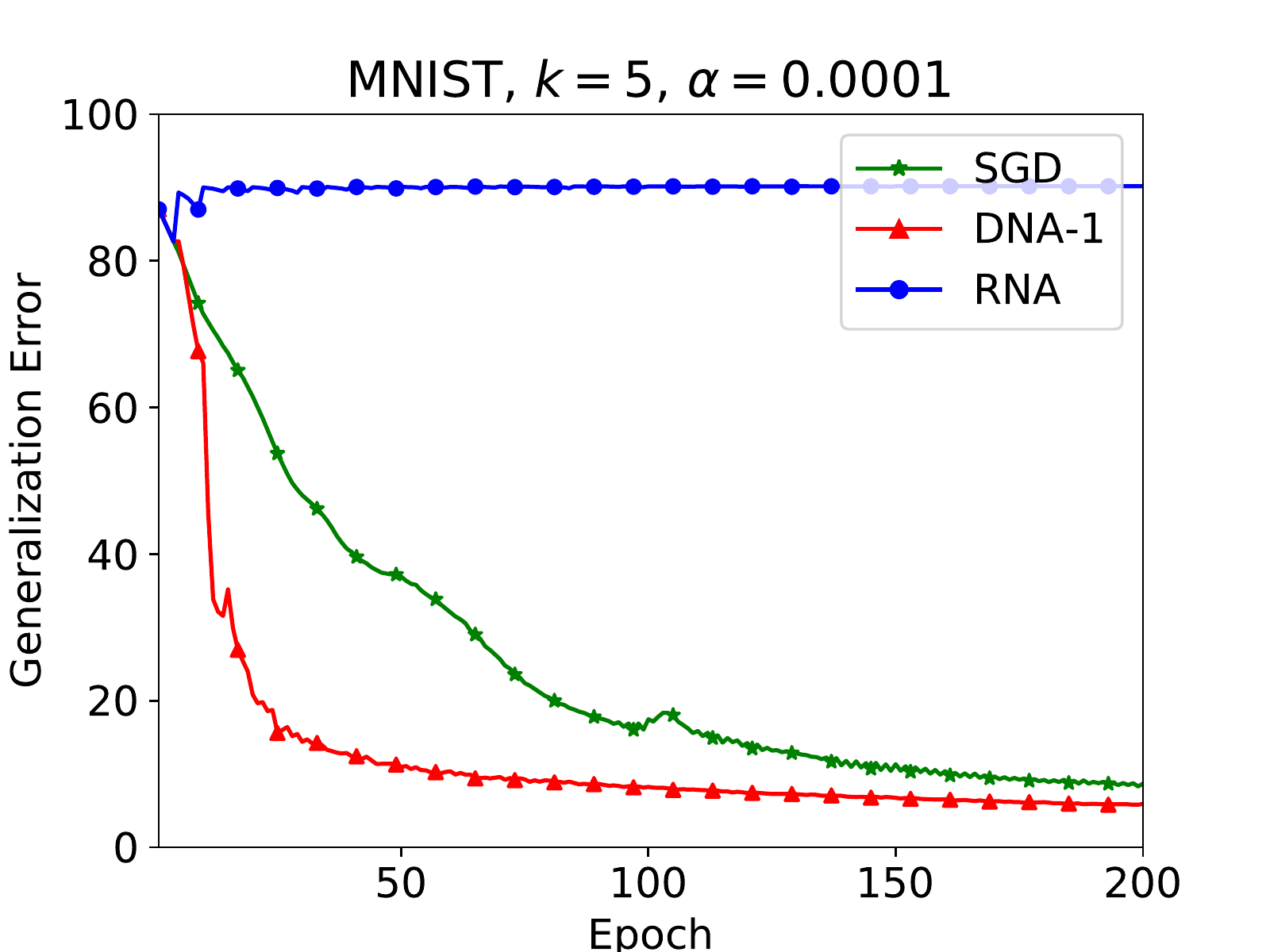}
				\caption{}
\end{subfigure}
\begin{subfigure}{0.3\textwidth}
		\centering
		\includegraphics[width=\textwidth]{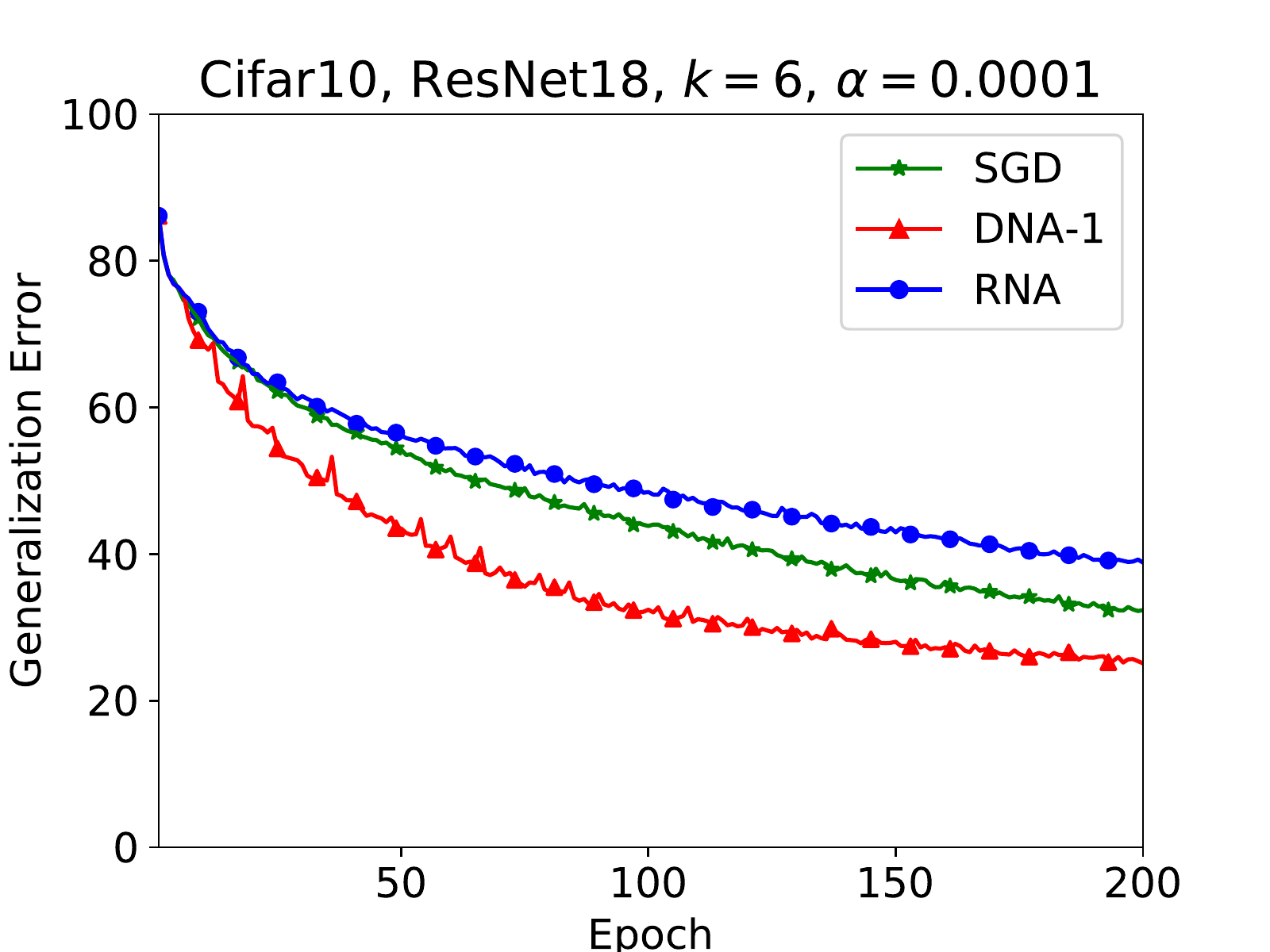}
		\caption{}
		\end{subfigure}
		\caption{\small{Accelerating neural-network training. (a) A 2-layer neural network with GD optimizer and fixed stepsize $0.0001$. (b) A 2-layer neural network with SGD optimizer and fixed stepsize $0.0001$. For both, we use $k=5$. (c) ResNet18 on CIFAR10 dataset with SGD optimizer and fixed stepsize $0.0001$. We use $k=6$. Note that these are not the best stepsize setting for the networks. Codebase {\tt Pytorch}. 
}}\label{fig:neu_net}
\end{figure*}

We call this  method DNA-2 (see Alg~\ref{alg:DNA-2}).
Note that $X^\top R+\lambda X^\top X$  can be singular, especially near the optimal solution. To remedy this, we propose either to add another regularization to the problem (\ref{eq:fkrnareg-}), or to consider a direct regularization on $c$ instead of $Xc$. We explain this next.

\subsection{DNA-3}
We set  $g(c)=\|c-e\|^2$ in \eqref{mainDNA} and consider a regularized version of  \eqref{eq:fkrna} as
\begin{eqnarray}\label{eq:fkrnareg}
\compactify \min_{\substack{c\in\R^{K+1}}}  f\left(Xc \right)+\frac{\lambda}{2}\|c-e\|^2,
 \end{eqnarray}
 where $\lambda>0$ and $e$ is a reference point for $c$.  By taking the derivative with respect to $c$ and setting it to 0, we find 
$
\compactify{ X^\top\nabla f(Xc)+\lambda (c-e)=0},
$
which after using the approximation \eqref{approximate_grad} becomes
$
 \compactify{X^\top(Rc+\nabla f(0))+\lambda X^\top(c-e)=0.} 
$
Therefore, $c^\star$ is given as a solution to the linear system: 
$
\compactify{(X^\top R+\lambda I)c=\lambda  e-X^\top \nabla f(0).}
$
We call this method DNA-3, and describe it in Alg~\ref{alg:DNA-3}.
\begin{algorithm}
	\SetAlgoLined
 	\SetKwInOut{Input}{Input}
	\SetKwInOut{Output}{Output}
     \SetKwInOut{Init}{Initialize}
     \SetKwInOut{Compute}{Compute}
\Input{Sequence of iterates $x_0,\ldots,x_{K+1}$; sequence of step sizes $\alpha_0,\ldots,\alpha_{K}$; regularizer $\lambda>0$; and $e\in\R^{k+1}$\;}
	
		\nl Set $R =\left[\frac{x_0 - x_{1}}{\alpha_0} -\nabla f(0),\ldots,\frac{x_K - x_{K+1}}{\alpha_K} -\nabla f(0)  \right]$ and $X = [x_0, \ldots,x_K]$\;
		\nl Set $c$ as a solution of the  linear system
		$(X^\top R+\lambda I) z = \lambda  e-X^\top \nabla f(0)$\;
	\Output{$x = \sum_{k = 0}^K c_k x_k$.}
 	\caption{ DNA-3 } \label{alg:DNA-3}
 \end{algorithm} 
\section{Numerical illustration}

We evaluate our techniques and compare against RNA and GD 
by using both synthetic data as well as real-world datasets. Overall, we find that DNA outperforms RNA in most settings by large margins. 

{\em Experimental setup.}
Our experimental setup comprises of 3 typical problems, least squares, ridge regression, and logistic regression, for which the optimal solution $x^\star$ is either known or can be evaluated using a numerical solver. We apply the online acceleration scheme in \citep{rna_16} and compare 3 versions of DNA against RNA and GD.
Our results show the difference between the functional values at the extrapolated point and at the optimal solution on a logarithmic scale (the lower the better), as the iterations progress.
The primary objective of our simulations is to show the effectiveness of DNA and its different versions to accelerate a converging, deterministic optimization algorithm. Therefore, we do not report any computation time of the algorithms and we do not claim these implementations are optimized. Note that the computation bottleneck of all algorithms (including RNA) is solving the linear system to calculate $c$, and because the dimensionality of the linear systems is the same in RNA and DNA, the extra cost is the same in both approaches. In our experiments,
we consider a fixed stepsize $\alpha_k=1/L$ for GD, where $L$ is the Lipschitz constant of $\nabla f$. 
We note that for DNA-1 and 2 we need to use the stepsize explicitly to construct $R$ as defined in Lem~\ref{lemma:dna1}. 

{\em Least Squares.}
We consider a least squares regression problem of the form
\begin{eqnarray}\label{ls}
\compactify \min_x f(x)\eqdef \frac{1}{2}\|Ax-y\|^2
\end{eqnarray}
where $A\in\mathbb{R}^{m\times n}$ with $m>n$ is the data matrix, $y\in\mathbb{R}^m$ is the response vector. For $m>n$ and ${\rm rank}(A)=n$, the objective function $f$ in \eqref{ls} is strongly convex. The optimal solution $x^\star$ to \eqref{ls} is given by
$
\compactify{x^\star=\argmin_x f(x)=(A^\top A)^{-1}A^\top y.}
$
For least squares we only consider the overdetermined systems, that is, $m>n$. 

{\em Ridge Regression.}
The classic ridge regression problem is of the form:
\begin{eqnarray}\label{ridge}
\compactify \min_x f(x)\eqdef \frac{1}{2}\|Ax-y\|^2+\frac{1}{2n}\|x\|^2,
\end{eqnarray}
where $A\in\mathbb{R}^{m\times n}$ is the data matrix, $y\in\mathbb{R}^n$ is the response vector.~The optimal solution $x^\star$ to \eqref{ridge} is given by
$
x^\star=\argmin_x f(x)=(A^\top A+\tfrac{1}{2n} I)^{-1}A^\top y.
$

{\em Logistic Regression.}
In logistic regression with $\ell_2$ regularization, the objective function $f(x)$ is the summation of $n$ loss function  of the form:
\begin{eqnarray}\label{lr}
\compactify f_i(x) = {\rm log}(1 +{\rm  exp}(-y_i\langle A(:,i), x \rangle)+\frac{1}{2m}\|x\|^2.
\end{eqnarray}
 We use the MATLAB function {\tt fminunc} to numerically obtain the minimizer of $f$ in this case.

{\em Synthetic Data.}
To compare the performance of different methods under different acceleration schemes, we are interested in the case where matrix $A$ has a known singular value distribution and we consider the cases where $A$ has varying condition numbers. We note that the condition number of $A$ is defined as $\kappa(A)\eqdef\nicefrac{\lambda_{\max}(A)}{\lambda_{\min}(A)}$, where $\lambda$ is the eigenvalue of $A$. We first generate a  random matrix and let $U\Sigma V^\top$ be its SVD. Next we create a vector $S\in\R^{\min\{m,n\}}$ with entries $s_i\in\R^+$ arranged in an nonincreasing order such that $s_1$ is maximum and $s_{\min\{m,n\}}$ is minimum. Finally, we form the test matrix $A$ as $A=U{\rm diag}(s_1\;\;s_2\;\cdots s_{\min\{m,n\}})V^\top$ such that $A$ will have a higher condition number  if $\nicefrac{s_1}{s_{\min\{m,n\}}}$ is large and smaller condition number if $\nicefrac{s_1}{s_{\min\{m,n\}}}$ is small. We create the vector $y$ as a random vector. 

{\em Real Data.} We use 15 different real-world datasets from the {\tt LIBSVM} repository \citep{libsvm}. 
We set apart the datasets with $y$-labels as $\{-1,1\}$ for Logistic regression and used the remaining 12 multi-label datasets for least squares and ridge regression problems. We use the matrix $A$ in its crude form, that is, without any scaling/normalizing or centralizing its rows or columns. 

{\em Acceleration Results.} We use GD as our baseline algorithm and the online acceleration scheme as explained in \citep{rna_16} for both RNA and DNAs to accelerate the GD iterates. 
For synthetic data (see Fig~\ref{fig:synthetic_ol16}), we see that for smaller condition numbers and for quadratic objective functions, DNA-1 and RNA has almost similar performance and DNA-2 and 3 show faster decrease of $f(x_k)-f(x^\star)$, but all of them are very competitive. As the condition number of the problems becomes huge, DNA-2 and 3 outperform RNA by large margins. However, for logistic regression problems we see performance gains for all versions of DNA compared to RNA. Though for huge condition numbers, for logistic regression problems, the performance of DNA-2 depends on the hyperparameter $\lambda$. We argue with experimental evidence as in Fig~\ref{fig:synthetic_ol16} that for huge condition numbers the sensitivity of the performance of DNA-2 is problem specific. We use an additional regularizer $\epsilon \|c\|^2$, where $\epsilon\approx 10^{-14}$ to find a stable solution to \eqref{solution2} of DNA-2.~Next, on real-world datasets in Figs~ \ref{fig:real_ls_ol16} and \ref{fig:realdata_rr_ol16}, we see that all versions of DNA outperform RNA, except in a few cases, where RNA and DNA-1 have almost similar performance. We indicate the oscillating nature of DNA-2 in some plots is due to its problem-specific sensitivity to the regularizer. In Fig~\ref{fig:realdata_lr_ol16}, we find for logistic regression problems on real datasets, DNA outperfoms RNA. We owe the success of DNA on non-quadratic problems to its adaptive gradient approximation. We note that the performance of all algorithms on the offline scheme of \citep{rna_16} are similar to online scheme of \citep{rna_16}. However, on the second online scheme used in \citep{rna_18}, all the algorithms perform extremely poorly. Therefore, we do not report the results in this paper. 

{\bf Application to the non-convex world: Accelerating neural network training.}
Modern deep learning requires optimization algorithms to work in a nonconvex setup. Although this is not the main goal of this paper, nevertheless, we implement our acceleration techniques for training neural networks and obtain surprisingly promising results.
We only use DNA-1 for experiments in this section. Tuning the hyperparameter $\lambda$ for the other versions of DNAs requires more time, and we leave this for future research. The {\tt Pytorch} implementation of RNA is based on \citep{rna_18}. Finally, see Fig~7
in Appendix for more results. 

{\em MNIST Classification.}
First, we trained a simple two-layer neural network classifier on MNIST dataset \citep{mnist} via GD and accelerate the GD iterates via the online scheme in \citep{rna_16} for both RNA and DNA-1. The two-layer neural network is wildely adopted in most tutorials that use MNIST dataset \footnote{https://github.com/pytorch/examples/blob/master/mnist/main.py}. In Fig~\ref{fig:neu_net} (a), DNA-1 gains acceleration by using GD iterates with a window size $k=5$. However, RNA fails to accelerate the GD iterates. This motivated us to train the same network on MNIST dataset classification \citep{mnist} via SGD as baseline algorithm and accelerate the SGD iterates via the online scheme in \citep{rna_16} for both RNA and DNA-1 (as in Fig~\ref{fig:neu_net} (b)). Again, with window size $k=5$, DNA-1 achieves better acceleration than RNA.  

{\em ResNet18 on CIFAR10.}
 Finally, we train the ResNet18 network \citep{resnet18} on CIFAR10 dataset \citep{cifar10} by SGD. Each epoch of SGD consists of multiple iterations and each iteration applies to $128$ training samples. The size of the training set is $5\times 10^4$ and the size of validation set is $10^4$. Each sample is a $32\times 32$ resolution color image and they are categorized into 10 classes. We accelerate the SGD iterates via the online scheme in \citep{rna_16} for both RNA and DNA-1. Again DNA-1 outperforms RNA in lowering the generalization error of the network (see Fig~\ref{fig:neu_net} (c)).


\bibliographystyle{plainnat}
{
	\bibliography{references}}



\clearpage
\appendix

\part*{Appendix}

\section{Anderson's Acceleration~\citep{Anderson}}\label{sec:anderson}

There are several acceleration techniques that have been proposed in the literature and they pose a lot of similarities. We quote the authors from \citep{brezinski} -- ``Methods for accelerating the convergence of various processes have been developed by researchers in a wide range of disciplines, often
without being aware of similar efforts undertaken elsewhere.'' In 1965 Anderson's acceleration was designed to accelerate Picard iteration for electronic structure computations. Because it is relevant in our current work, we give a brief description of it for completeness.
 
 For a given sequence of iterate $\{x_k\}$ with $x_k\in\R^n$ and a mapping $\Phi(\cdot): \R^n \rightarrow \R^n$, the fixed-point algorithm generates a recursive update of the iterates as:
 \begin{eqnarray}\label{FP}
 x_{k+1}=\Phi (x_k).
 \end{eqnarray}
Let there be $m_{k}+1$ evaluations of the fixed point map $\phi$. Anderson's acceleration  technique computes a new iteration as a linear combination of the previous $m_{k}+1$ evaluations. We explain it formally in Alg~\ref{alg_1}. In Alg~\ref{alg_1}, $m$ is considered as a hyperparameter that sets the quantity $m_k$ as $\min\{m,k\}$, where $k$ is the iteration counter and $m$ is known as the depth.~This is used to determine the window size to compute $\hat{c}$--the coefficients for linear combination of the fixed point evaluations. In other words,  in each iteration, by solving the optimization problem:
 $$\boxed{\hat{c}^{(k)}=\arg\min_{c}\|F^k c\|\quad {\rm subject\;\;to}\quad\sum_ic_i=1,}$$
one can obtain the extrapolation coefficients  $\hat{c}^{(k)}$ that help to determine the accelerated point $x_{k+1}.$ Toth and Kelley pointed out that, in principle, any norm can be used in the minimization step \citep{tothkelley}. 
 \begin{algorithm}
	\SetAlgoLined
 	\SetKwInOut{Input}{Input}
	\SetKwInOut{Output}{Output}
     \SetKwInOut{Init}{Initialize}
     \SetKwInOut{Compute}{Compute}
\nl\Input{$x_0\in\R^n$ and $m\geq 1$\;}
     \nl\Init {Set $x_1=\Phi(x_0), m_k=\min\{m,k\}, F^k=(f_{k-m_k},f_{k-m_k+1},\cdots, f_k)\in\R^{n\times (m_k+1)}$, where $f_i=\Phi(x_i)-x_i$\;}
	 \nl \For{$k=1,2,\cdots$}
	{
		\nl Find $\hat{c}^{(k)}\in\R^{(m_k+1)}$ such that: $\hat{c}^{(k)}=\arg\min_{\alpha}\|F^kc\|$ subject to $\sum_ic_i=1$\;
		\nl Set $x_{k+1}=\sum_{i=0}^{m_k}\hat{c}_i^{(k)}\Phi(x_{k-m_k+i}).$
 	}
	
 	\caption{Anderson Acceleration} \label{alg_1}
 \end{algorithm}
The summability of the coefficients $c_i$ or the {\em normalization condition} was not explicitly mentioned in the original work of Anderson. Because $c_i$'s can be determined up to a multiplicative scalar, one can impose the {\em normalization condition}. However, it does not restrict generality.  
We refer the readers to~\citep{kelley,tothkelley,walker,Anderson,brezinski} for a comprehensive idea of Anderson's acceleration technique.

 \section{Acceleration Schema}\label{sec:schema}
In this section, we explain different acceleration schema used by Scieur et al. in \citep{rna_16,rna_18} for completeness. 
\begin{remark}
For GD, the updates of the iterates are done via the simple update rule \eqref{GD}, which is explained in Fig~\ref{fig:accln}(a). 
\end{remark}

\subsubsection{Online Scheme 1}
We explain the acceleration scheme proposed in \citep{rna_16} herein. First, we run $k$ iterations of GD to produce the sequence of iterates $\{x_i\}_{i=1}^k$ and then use extrapolation to generate a new point $x_k'$.We use $x_k'$ as the initial point of GD and produce a set of next $k$ iterates via GD. At this end, we further use the extrapolation scheme to produce a second offline update $x_{k+1}'$ which is used as next the initial point of GD, and this process continues. See Fig~\ref{fig:accln}(b). 

\subsubsection{Online Scheme 2}
The acceleration scheme proposed in \citep{rna_18}, is more involved than the one propsed in \citep{rna_16}. First, we run $k$ iterations of GD to produce a sequence of iterates $\{x_i\}_{i=1}^k$ and then use extrapolation to generate $x_k'$. Next, we use $x_k'$ as the starting point of GD to produce $x_{k+1}$. Now, we start from the second iterate $x_2$ and consider a set of $k$ iterates $\{x_2,x_3,\cdots, x_k,x_{k+1}\}$ to produce the second offline update $x_{k+1}'$ via extrapolation which is to be used as the next starting point of GD, and this process continues. See Fig~\ref{fig:accln}(c). 

\subsubsection{Offline scheme}
Lastly, we describe an offline update scheme, as illustrated in Fig~\ref{fig:accln}(d). First, we run the GD to produce the sequence of iterates $\{x_k\}$ and then use the acceleration on the set of first $k$ iterates to produce the first offline update $x_k'$ and concatenate it with the previous $(k-1)$ GD updates. Next, we start from the second iterate $x_2$ and consider a set of $k$ iterates to produce the second offline update $x_{k+1}'$ via acceleration and this process continues. As a result, the offline accelerated updates are generated as $\{x_1, x_2\cdots, x_k',x_{k+1}',\cdots\}$.

 \begin{figure*}
  \includegraphics[width=\textwidth]{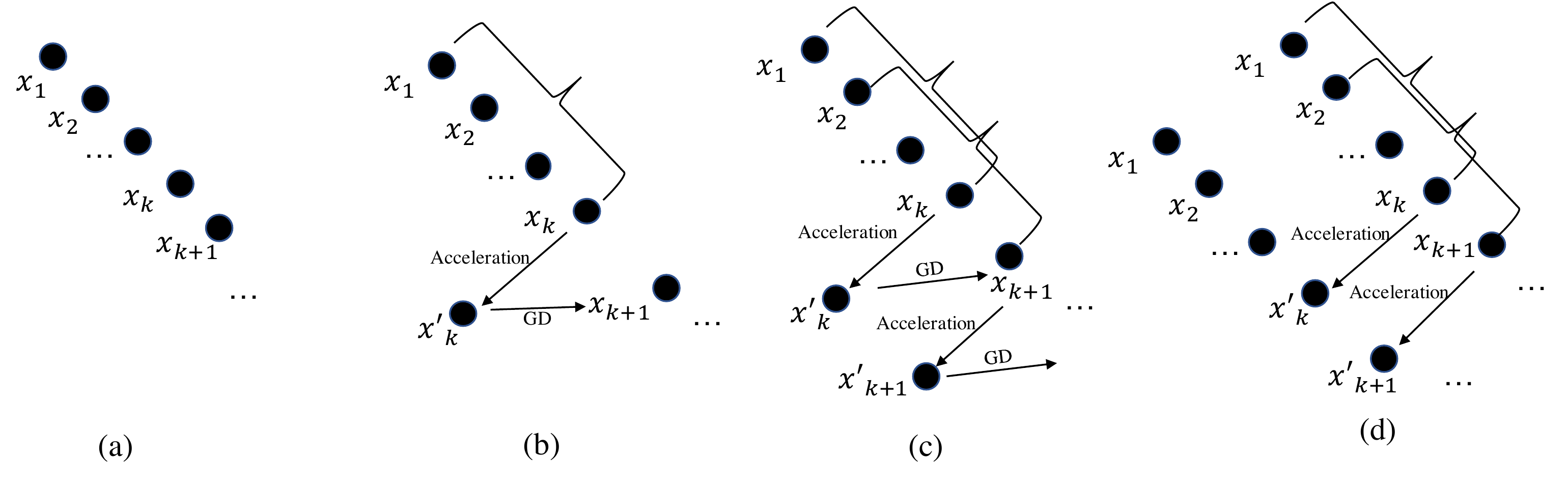} 
   \caption{Updates via: (a) gradient descent, (b) online extrapolation on gradient descent \citep{rna_16}, (c) online extrapolation on gradient descent \citep{rna_18}, and (d) offline scheme.} 
   \label{fig:accln}
 \end{figure*}

\begin{proof}{\it of Lemma \ref{lemma:dna1}.}
Let $h(c) = f(Xc)$, 
from the first order optimality condition we have 
  \begin{eqnarray*}
  \nabla h(c) &=& X^\top \nabla f\left(Xc \right) = 0
  \end{eqnarray*}
  For quadratic objective function the gradient is affine, i.e 
 \begin{eqnarray*}
 \nabla f(Xc) &=& AXc + \nabla f(0)\\
 &=&  \sum_{k = 0}^K c_k Ax_k + \nabla f(0)\\
&=& \sum_{k = 0}^K c_k \left( \nabla f(x_k) -  \nabla f(0)\right) + \nabla f(0).  
 \end{eqnarray*} 
   By using the relation between the iterates of GD method we find $\nabla f\left(x_k \right) =  \frac{x_k - x_{k+1}}{\alpha_k}.$ Hence $\nabla f(Xc) =Rc + \nabla f(0).$
   By injecting this in the first order optimality condition we get the result.
\end{proof}  
\begin{proof}{\it of Lemma \ref{lemma:krna}.}

Since $f$ is quadratic then $f(x) = f(x^\star) + \|x-x^\star\|_H^2$. Therefore, from the definition of $c_R$ and $c_D$ and using Proposition 2.2 in \citep{rna_16} we conclude the result. 
\end{proof} 
 \begin{proof}{\it of Lemma \ref{lemma:ckrna_lemma}.}

  The Lagrangian of the problem \eqref{eq:rna} is 
  $$L(c,\lambda) = h(c) + \lambda \left(\sum_{k = 0}^K c_k -1\right),$$
  where $\lambda>0$ is the Lagrange multiplier.
  The first order optimality conditions are 
  \begin{eqnarray}
  \nabla L_x(x,\lambda) &=& X^\top \nabla f\left(Xc \right) + \lambda \mathbbm{1} = 0\\
  \nabla L_\lambda(x,\lambda) &=& c^\top \mathbbm{1} - 1= 0.
  \end{eqnarray}
  For quadratic objective functions, the gradient is affine and because $\sum_{k = 0}^K c_k=1$ we have  
\begin{equation} \label{eq:lingrad}
  \nabla f\left(Xc \right) =  \sum_{k = 0}^K c_k \nabla f\left(x_k \right).
  \end{equation} 
 By using the relation between the iterates of GD method we find $$\nabla f\left(x_k \right) =  \frac{x_k - x_{k+1}}{\alpha_k}.$$
 By using the above expression in equation (\ref{eq:lingrad}) we further get 
 \begin{equation} \label{eq:lingrad-}
  \nabla f\left(Xc \right) =  \sum_{k = 0}^K c_k \frac{x_k - x_{k+1}}{\alpha_k} = \tilde{R}c.
  \end{equation} 
Substituting (\ref{eq:lingrad-}) in the first optimality condition and solving for $c$ we get 
  $$
   c   = - \lambda  \left(X^\top \tilde{R} \right)^{-1}\mathbbm{1}.$$
Next we use it in the second optimality condition and solve it for $\lambda$ to find
   $$ 
  \lambda =  \frac{-1}{\mathbbm{1}^\top   \left(X^\top\tilde{R} \right)^{-1}\mathbbm{1}},$$
 and therefore the final expression for $c$ is 
 $$
  c = \frac{ \left(X^\top \tilde{R} \right)^{-1}\mathbbm{1}}{\mathbbm{1}^\top   \left(X^\top \tilde{R} \right)^{-1}\mathbbm{1}}.
  $$
 \end{proof}
 
 \section{Example with Quadratic function.}
Let $f(x)=\frac{1}{2}x^\top Ax$, where $A$ is symmetric and positive definite. We know $\nabla f(x)=Ax.$ By using the extrapolation we find the coefficients $c_i$s such that $x=\sum_ic_ix_i=Xc$, where $X=[x_0\;\;x_1\;\cdots x_k]$ is a matrix generated by stacking $k$ iterates as its column and $c\in\R^k$ is a vector of coefficients. 
We know 
$$
f(Xc)=\frac{1}{2}c^\top X^\top AXc, \qquad\text{ for DNA} \qquad c_{\rm D}= 0, \qquad\text{and for DNA-1} \qquad c_{\rm D1}=\frac{z}{\mathbbm{1}^\top z}\text{ where } z=(X^\top \tilde{R})^{-1}\mathbbm{1}.
$$
Therefore, we find
$$
f_D = 0, ~ {\text{ and }} 
f_{\rm D1}=\frac{\mathbbm{1}^\top (\tilde{R}^\top X)^{-1}X^\top AX (X^\top \tilde{R})^{-1}\mathbbm{1}}{2(\mathbbm{1}^\top z)^2},
$$
which for $\tilde{R}= [\nabla f(x_0)\;\;\nabla f(x_1)\;\cdots \nabla f(x_k)]=[Ax_0\;\;Ax_1\;\cdots Ax_k]= AX$ further reduces to 
$$
f_{\rm D1}=\frac{1}{2 \mathbbm{1}^\top (X^\top A X)^{-1} \mathbbm{1}}. 
$$
Similarly, we find for RNA
$$
c_{\rm R}=\frac{(\tilde{R}^\top \tilde{R})^{-1}\mathbbm{1}}{\mathbbm{1}^\top (\tilde{R}^\top \tilde{R})^{-1} \mathbbm{1}}.
$$
Therefore,  
$$
f_{\rm R}=\frac{\mathbbm{1}^\top (\tilde{R}^\top \tilde{R})^{-1}X^\top A X (\tilde{R}^\top \tilde{R})^{-1}\mathbbm{1}}{2(\mathbbm{1}^\top (X^\top A^2X)^{-1} \mathbbm{1})^2},
$$
which further reduces to  
$$
f_{\rm RNA}=\frac{\mathbbm{1}^\top (X^\top A^2 X)^{-1}X^\top A X (X^\top A^2 X)^{-1}\mathbbm{1}}{2(\mathbbm{1}^\top (X^\top A^2X)^{-1} \mathbbm{1})^2}.
$$
In order to prove Lemma \ref{lemma:ratio} we need the following  Lemma. 
\begin{lemma}\label{pinv}
If the sequence of iterates $\{x_k\}$ are linearly independent then we have:\\
(i) the matrices $AX$ and $A^{\frac{1}{2}} X$ have full column ranks.\\
(ii) $(A^{1/2}X)^\dagger A^{-1/2}((AX)^\dagger)^\top=(X^\top A^2 X)^{-1}.$
\end{lemma}

\begin{proof}
(i) Since $A$ is symmetric and positive definite, ${\rm rank}(A)= {\rm rank}(A^{1/2})=n$. As the iterates $\{x_k\}$ are linearly independent, $X = [x_0, \ldots,x_K]\in\R^{n\times (K+1)}$ has full column rank. Therefore, the matrices $AX$ and $A^{\frac{1}{2}} X$ have full column ranks. 

(ii) We know if a matrix $B$ is of full column rank then $B^\dagger = (B^\top B)^{-1}B^\top.$ By using the above and (i) and we find
\begin{eqnarray*}
(A^{1/2}X)^\dagger A^{-1/2}((AX)^\dagger)^\top&=&(X^\top A^{1/2}A^{1/2}X)^{-1}X^\top A^{1/2}A^{-1/2}((X^\top AAX)^{-1}X^\top A)^\top\\
&=&(X^\top AX)^{-1}X^\top((X^\top A^2X)^{-1}X^\top A)^\top\\
&\overset{(X^\top A^2X)^\top=X^\top A^2X}=&(X^\top AX)^{-1}(X^\top AX)(X^\top A^2X)^{-1}\\
&=&(X^\top A^2 X)^{-1}.
\end{eqnarray*} 
Hence the result. 
\end{proof}
\begin{figure*}
\centering
	\begin{subfigure}{0.4955\textwidth}
		\centering
		\includegraphics[width=\textwidth]{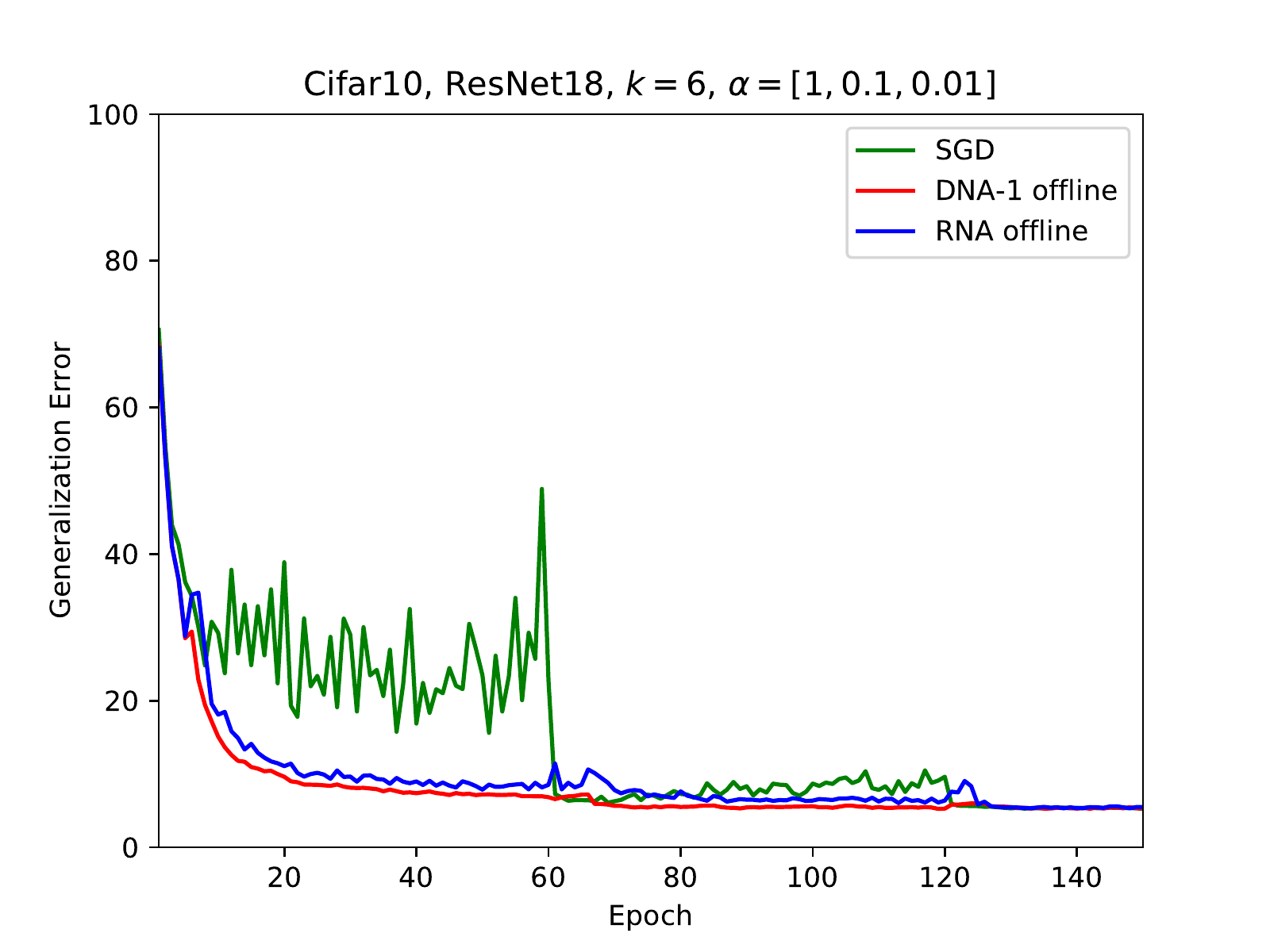}
				\caption{}
\end{subfigure}
\begin{subfigure}{0.4955\textwidth}
		\centering
		\includegraphics[width=\textwidth]{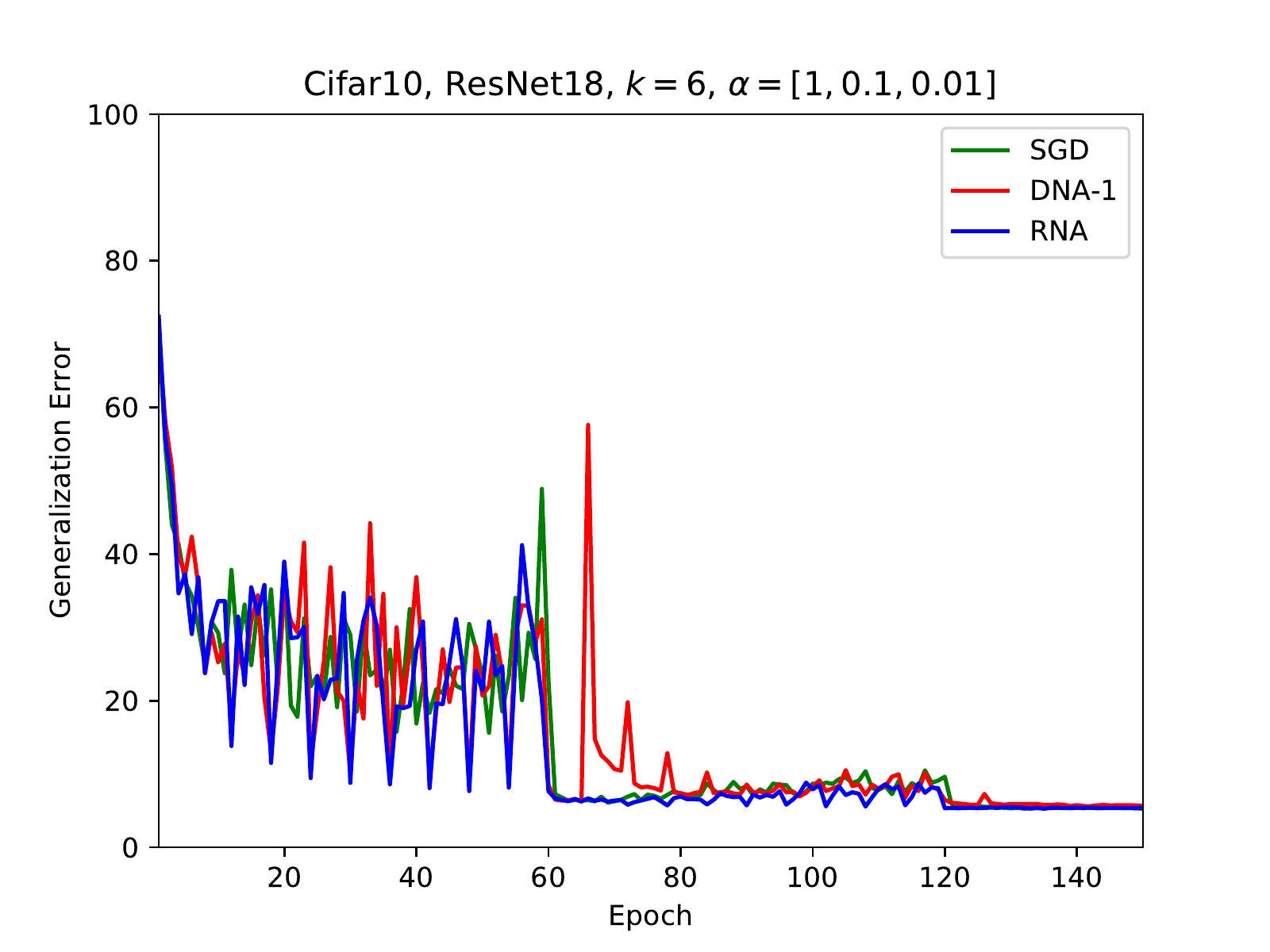}
				\caption{}
\end{subfigure}
\caption{\small{Acceleration on Neural Network. (a)  Experiment implementing ResNet18 on Cifar10 dataset with SGD as training algorithm with decaying stepsize across the epochs. For both DNA-1 and RNA, the window size is set to $k=6$ and we use the offline scheme of \citep{rna_16}. (b) Experiment implementing ResNet18 on Cifar10 dataset with SGD as training algorithm with decaying stepsize across the epochs. For both DNA-1 and RNA, the window size is set to $k=6$ and we use the online scheme of \citep{rna_16}. In both cases, RNA and DNA-1 fail to accelerate the SGD iterates. This indicates the fact that the stepsize is a very important hyperparameter and one needs to further explore it in case of accelerating a neural network training. 
}}\label{fig:appendix_neu_net}
\end{figure*}

\begin{proof}{\it of Lemma \ref{lemma:UB}.}
We have $\tilde{R}= AX$. Since $A$ is symmetric and positive definite, it is invertible and $X=A^{-1}R.$ Set $y=((A^{1/2}X)^\dagger)^\top\mathbbm{1}$. 
 and let $\tilde{R}^\dagger$ be the pseudo-inverse of $R$. Therefore, $\tilde{R}^\dagger$ can be computed as
$$
\tilde{R}^\dagger = (\tilde{R}^\top \tilde{R})^{-1}\tilde{R}^\top,
$$
and $(\tilde{R}^\dagger)^\top$ is 
$$
(\tilde{R}^\dagger)^\top = \tilde{R}(\tilde{R}^\top \tilde{R})^{-1}.
$$
We also note that $\tilde{R}^\dagger \tilde{R} = I_k$ and $\tilde{R}^\top(\tilde{R}^\dagger)^\top = I_k$, where $I_k$ is an identity matrix of size $k$. Therefore, we have 
\begin{eqnarray*}
2f_{\rm D1} &=& \frac{1}{\mathbbm{1}^\top (X^\top A X)^{-1}\mathbbm{1}}\\
&=& \frac{1}{\mathbbm{1}^\top (X^\top A^{1/2} A^{1/2} X)^\dagger\mathbbm{1}},
\end{eqnarray*}
Since, $(X^\top A^{1/2})^\top = A^{1/2}X$, by using the property of pseudo-inverse, we can write $$(X^\top AX)^{-1}=(X^\top A^{1/2}A^{1/2} X)^\dagger  = (A^{1/2}X)^\dagger((A^{1/2}X)^\dagger)^\top$$ and the above expression becomes 
\begin{eqnarray}\label{f_krna}
2f_{\rm D1} &=& \frac{1}{\mathbbm{1}^\top(A^{1/2}X)^\dagger((A^{1/2}X)^\dagger)^\top\mathbbm{1}}\overset{y=((A^{1/2}X)^\dagger)^\top\mathbbm{1}}=\frac{1}{y^\top y}=\frac{1}{\|y\|_2^2}.
\end{eqnarray}
Similarly we find, $(X^\top A)^\top = AX$, and again by using the property of pseudo-inverse, we can write $$(X^\top A^2X)^{-1}=(X^\top AAX)^\dagger  = (AX)^\dagger((AX)^\dagger)^\top$$ and 
\begin{eqnarray*}
2f_{\rm R} &=&\frac{\mathbbm{1}^\top (AX)^\dagger((AX)^\dagger)^\top X^\top A X (AX)^\dagger((AX)^\dagger)^\top\mathbbm{1}}{(\mathbbm{1}^\top (AX)^\dagger((AX)^\dagger)^\top \mathbbm{1})^2}\\
&\overset{AX=\tilde{R}}=&\frac{\mathbbm{1}^\top \tilde{R}^\dagger(\tilde{R}^\dagger)^\top(\tilde{R}^\dagger)^\top A^{-1} \tilde{R} \tilde{R}^\dagger(\tilde{R}^\dagger)^{\top}\mathbbm{1}}{(\mathbbm{1}^\top \tilde{R}^\dagger (\tilde{R}^\dagger)^{\top} \mathbbm{1})^2}\\
&\overset{\tilde{R}^\dagger = \tilde{R}^\dagger(\tilde{R}^\dagger)^\top(\tilde{R}^\dagger)^\top, (\tilde{R}^\dagger)^\top = \tilde{R} \tilde{R}^\dagger(\tilde{R}^\dagger)^{\top}}=&\frac{\mathbbm{1}^\top \tilde{R}^\dagger A^{-1} (\tilde{R}^\dagger)^\top\mathbbm{1}}{(\mathbbm{1}^\top  (\tilde{R}^\top \tilde{R})^{-1} \mathbbm{1})^2}\\
&\overset{z\eqdef \tilde{R}^\dagger\mathbbm{1}}{=}&\frac{z^\top A^{-1} z} {(z^\top z)^2}.
\end{eqnarray*}
Therefore, \begin{eqnarray}\label{f_rna}
2f_{\rm R} =\frac{z^\top A^{-1} z}{(z^\top z)^2}=\frac{\|z\|_{A^{-1}}^2}{\|z\|_2^4}.
\end{eqnarray}
Combining \eqref{f_krna} and \eqref{f_rna} we obtain the ratio between $f_R$ and $f_D$. 

From lemma~\ref{pinv} we have $y^\top A^{-1/2} z = z^\top z$ then by uisng Cauchy Swartz inequality we conclude that ${\|z\|_{A^{-1}}\|y\|_2} \ge y^\top A^{-1/2} z= {\|z\|_2^2}$
whence $\frac{f_{\rm R}}{f_{\rm D1}} \ge 1$.
\end{proof}
\begin{proof}{\it of Theorem \ref{lemma:UB}.}
Recall that $\frac{f_{{\rm R}}}{f_{{\rm D1}}}=\frac{\|z\|_{A^{-1}}^2\|y\|_2^2}{\|z\|_2^4}.$ Also recall that 
$z\eqdef (\tilde{R}^\dagger)^\top \mathbbm{1} = ((AX)^\dagger)^\top \mathbbm{1}$ and $y\eqdef((A^{1/2}X)^\dagger)^\top\mathbbm{1}$. 
Therefore, $A^{1/2}z$  is the minimum norm solution to the linear system: $X^\top A^{1/2}A^{1/2}z=\mathbbm{1}$
and similarly, $y$  is the minimum norm solution to the linear system: $X^\top A^{1/2}y=\mathbbm{1}$. By using the above fact, we find $\|y\|_2\leq \|A^{1/2}z\|_2=\|z\|_A$ and  we can rewrite the ratio as:
\begin{eqnarray}\label{ratio_ub}
\frac{f_{{\rm R}}}{f_{{\rm D1}}}=\frac{\|z\|_{A^{-1}}^2\|y\|_2^2}{\|z\|_2^4}\overset{y=A^{1/2}z}=\frac{\|z\|_{A^{-1}}^2\|z\|_A^2}{\|z\|_2^4}.
\end{eqnarray}
From \eqref{ratio_ub} the quantity $\displaystyle{\max_{z\neq 0} \frac{\|z\|_{A^{-1}}^2\|z\|_A^2}{\|z\|_2^4}}$ is equivalent to $\displaystyle{\max_{\|z\|_2=1} {\|z\|_{A^{-1}}^2\|z\|_A^2}}\leq \displaystyle{\max_{\|z\|_2=1}\|z\|_{A^{-1}}^2\max_{\|z\|_2=1} \|z\|_A^2} \eqdef U_R.$

Note that $U_R \leq {\lambda_{\max}(A^{-1})}\lambda_{\max}(A)= \frac{\lambda_{\max}(A)}{\lambda_{\min}(A)}=\kappa(A).$

Let $U \Sigma U^\top$ be an eigenvalue decomposition of $A$ then 
$U_R = \displaystyle{\max_{\|z\|_2=1}\|z\|_{\Sigma^{-1}}^2\max_{\|z\|_2=1} \|z\|_\Sigma^2}$. By considering a vector with $1/\sqrt{2}$ at the first and last position and zero everywhere else we conclude that 
\begin{eqnarray*}
U_R &\ge& \left( \frac{1}{2 \lambda_{\max}(A)} + \frac{1}{2 \lambda_{\min}(A)}  \right)\left(  \frac{1}{2 \lambda_{\min}(A)} + \frac{1}{2 \lambda_{\max}(A)}   \right)\\
& \ge & \frac{2+\kappa(A)}{4}.
\end{eqnarray*}
\end{proof}

\section{Reproducible research}
See the {\tt LIBSVM} dataset from the repository online: \url{https://www.csie.ntu.edu.tw/~cjlin/libsvmtools/datasets/}.
See the source code of RNA from: \url{https://github.com/windows7lover/RegularizedNonlinearAcceleration}. For {\tt MATLAB} and {\tt Pytorch} code that is used to produce all the results for our DNA, please email the authors. 

\end{document}

%% file: everything.tex
\usepackage{layout}
\usepackage{multirow}

\usepackage{amsfonts}
\usepackage{amsmath}
\usepackage{amsthm}
\usepackage{amssymb} 

\usepackage{mdframed} 
\usepackage{thmtools}

\definecolor{shadecolor}{gray}{0.95}
\declaretheoremstyle[
headfont=\normalfont\bfseries,
notefont=\mdseries, notebraces={(}{)},
bodyfont=\normalfont,
postheadspace=0.5em,
spaceabove=1pt,
mdframed={
  skipabove=8pt,
  skipbelow=8pt,
  hidealllines=true,
  backgroundcolor={shadecolor},
  innerleftmargin=4pt,
  innerrightmargin=4pt}
]{shaded}

 \usepackage{xcolor}
\usepackage{color}
\usepackage{graphicx}

\usepackage{verbatim}

\newcommand{\R}{\mathbb{R}} 

\newcommand{\cO}{{\cal O}}

\usepackage[colorinlistoftodos,bordercolor=orange,backgroundcolor=orange!20,linecolor=orange,textsize=scriptsize]{todonotes}

\newcommand{\eqdef}{:=} 

\DeclareMathOperator{\argmin}{argmin}        

\newtheorem{lemma}{Lemma}

\newtheorem{theorem}{Theorem}
\newtheorem{proposition}{Proposition}

\theoremstyle{definition}

\theoremstyle{remark}
\newtheorem{remark}{Remark} 
\usepackage{bbm}